\documentclass[12pt,reqno]{article}

\title{Annihilating branching Brownian motion}

\author{
Daniel Ahlberg\thanks{Department of Mathematics, Stockholm University, 
\href{mailto:daniel.ahlberg@math.su.se}{daniel.ahlberg@math.su.se}}
\and 
Omer Angel\thanks{Department of Mathematics, University of British Columbia, 
\href{mailto:angel@math.ubc.ca}{angel@math.ubc.ca}}
\and 
Brett Kolesnik\thanks{Department of Statistics, University of Oxford, 
\href{mailto:brett.kolesnik@stats.ox.ac.uk}{brett.kolesnik@stats.ox.ac.uk}}
}

\date{}

\usepackage[T1]{fontenc}
\pdfoutput=1
\usepackage{microtype,hyperref,graphicx,dsfont}
\usepackage{amsthm,amsmath,amssymb}
\usepackage{enumitem}
\usepackage[nameinlink,capitalise]{cleveref}
\usepackage[font=footnotesize,margin=2cm]{caption}
\usepackage{cite}
\usepackage{todonotes}

\makeatletter 
\renewcommand\@biblabel[1]{#1.} 
\makeatother

\usepackage[nameinlink]{cleveref}
\crefname{theorem}{Theorem}{Theorems}
\crefname{thm}{Theorem}{Theorems}
\crefname{mainthm}{Theorem}{Theorems}
\crefname{lemma}{Lemma}{Lemmas}
\crefname{lem}{Lemma}{Lemmas}
\crefname{remark}{Remark}{Remarks}
\crefname{claim}{Claim}{Claims}
\crefname{subclaim}{Sub-claim}{Sub-claims}
\crefname{prop}{Proposition}{Propositions}
\crefname{proposition}{Proposition}{Propositions}
\crefname{defn}{Definition}{Definitions}
\crefname{corollary}{Corollary}{Corollaries}
\crefname{conjecture}{Conjecture}{Conjectures}
\crefname{question}{Question}{Questions}
\crefname{chapter}{Chapter}{Chapters}
\crefname{section}{Section}{Sections}
\crefname{figure}{Figure}{Figures}
\newcommand{\crefrangeconjunction}{--}

\theoremstyle{plain}

\newtheorem{thm}{Theorem}
\newtheorem*{thm*}{Theorem}
\newtheorem{lemma}[thm]{Lemma}

\newtheorem{prop}[thm]{Proposition}

\newtheorem{question}[thm]{Question}
\theoremstyle{definition}

\theoremstyle{remark}
\newtheorem*{remark}{Remark}
\numberwithin{equation}{section}

\renewcommand{\P}{\mathbb P}
\newcommand{\Z}{\mathbb Z}
\newcommand{\E}{\mathbb E}
\newcommand{\R}{\mathbb R}

\newcommand{\Q}{\mathbb Q}

\newcommand{\eps}{\varepsilon}

\newcommand{\cB}{\mathcal B}
\newcommand{\cC}{\mathcal C}

\newcommand{\cG}{\mathcal G}

\begin{document}

\maketitle

\begin{abstract}
We study an interacting system of competing particles on the real line.
Two populations of positive and negative particles evolve according to  
branching Brownian motion.
When opposing particles meet, 
their charges neutralize and the particles annihilate, 
as in an inert chemical reaction.
We show that, with positive probability, the two populations coexist and that, 
on this event, the interface is asymptotically linear with a random slope.
A variety of generalizations and open problems are discussed. 
\end{abstract}

\maketitle

\section{Introduction}

{\bf Branching Brownian motion} (BBM) is a classical 
model for population growth with diffusion, see, e.g., \cite{MK75,B78,B83,LS87}. 

In this work, we study a generalization of BBM, 
consisting of an 
interacting system of competing particles.
As in standard BBM, independent Brownian particles split dyadically 
at unit rate. Each particle is of one of two types (or colors), which we 
refer to as {\bf positive} and {\bf negative} (or, at times, red and blue).
The key novelty is that opposing particles annihilate upon contact. 
We call this process {\bf annihilating branching Brownian motion} (ABBM).

\begin{figure}[h]
  \centering
 \includegraphics[width=.5\textwidth]{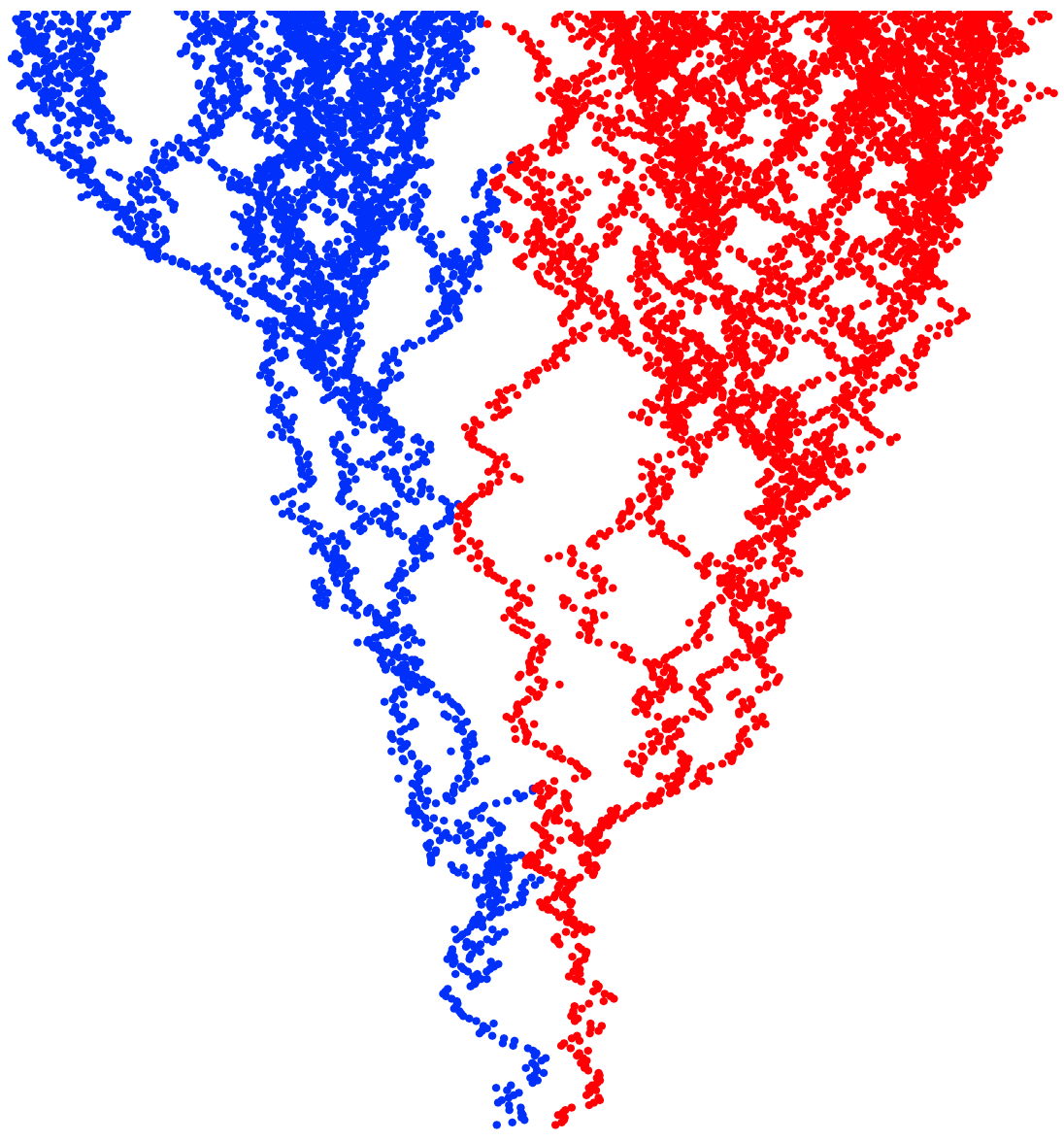}
  \caption{ABBM started with particles at $\pm1$.
    Each horizontal slice is a configuration of particles at a moment in time.}
\end{figure}

With only one type (and hence no annihilation) ABBM reduces to 
classical BBM.
We also note that ABBM naturally generalizes to any number of types,
in which any two particles of different type annihilate upon contact.
If there are initially $k$ types in total, we call it the $k$-type 
ABBM. 

It is possible that one or more of the types 
are annihilated eventually.
It is less obvious whether there is a positive probability 
of \textbf{coexistence}, i.e., the event that all colors initially present 
remain present in the system at all times.
We answer this positively. 

\begin{thm}\label{thm:multi_coex}
Let $k\ge2$. 
For a $k$-type ABBM started from a finite, 
non-trivial initial configuration, 
there is a positive probability of coexistence. 
\end{thm}

The non-triviality assumption mentioned in the above theorem, 
which we will use throughout, is very simple:
A configuration of multi-colored particles in $\R$ 
is called \textbf{non-trivial} if no two particles of distinct colors 
are at the same location.
Some form of non-triviality is obviously needed, 
since otherwise annihilation occurs immediately 
and some (or even all) colors might then
 die out deterministically at time 0.
Also, naturally, a configuration is \textbf{finite} 
if it consists of a finite number of particle.

Next, we consider ABBM with two types, 
and focus on the interface separating the two types of particles 
on the event of coexistence.
We shall refer to a configuration of particles 
as {\bf ordered} if 
no two particles of the same type have a particle of a different type positioned between them. 
(Equivalently, there is a partition of $\R$ into disjoint intervals 
so that each color occupies only one interval, 
and different colors occupy different intervals.)
In the case of two types, a configuration of positive and negative 
particles is ordered if the rightmost negative particle lies 
to the left of the leftmost positive particle
(or vice-versa, 
though by symmetry we may assume the order is as described).
The simplest example of an ordered configuration 
is that of a single negative particle 
and a single positive particle to its right.

Note that, since the diffusion mechanism 
of the Brownian particles is a continuous motion, 
and since the branching mechanism produces particles 
at the position of the parent, 
it follows that ABBM is {\bf order preserving}. 
That is, 
if ABBM is started from an ordered non-trivial initial configuration, 
then the configuration of particles will remain ordered at all future times. 

Consider a two-type ABBM started from a finite, 
non-trivial and ordered configuration.
Suppose, without loss of generality, 
that all negative particles are positioned to the left of all positive particles. 
Then this property will continue to hold at all future times.
Let $I_-(t)$ denote the location of the rightmost negative particle 
and $I_+(t)$ the location of the leftmost positive particle at time $t\ge0$.
(As usual, the supremum and infimum of the empty set 
are $-\infty$ and $+\infty$, respectively.)
Thus the interval $(I_-(t),I_+(t))$ 
is the maximal empty interval separating the 
two types of particles at time $t$.
We define the {\bf interface} between 
the two types as the midpoint 
\[
I(t) := \frac{I_-(t)+I_+(t)}{2}
\]
between the opposing sets of particles. 
Note that we have $I(t)=\pm\infty$ 
for all large $t$ if one of the types dies out, 
and is undefined if both do.
However, our main interest is in the behavior of $I(t)$ 
on the event of coexistence.
Our next result establishes the existence 
of an asymptotic limiting speed for the interface on this event. 

\begin{thm}\label{thm:interface}
  Consider a two-type ABBM started from a finite, non-trivial and ordered initial configuration.
  On the event of coexistence, the limit
  \[
    \lambda_*:= \lim_{t\to\infty} \frac{I(t)}{t}
  \] 
  exists almost surely.
  Moreover, the law of the limiting speed $\lambda_*$, conditioned on coexistence, has support $[-\sqrt{2},\sqrt{2}]$ and no atoms.
\end{thm}

\begin{figure}
\centering
\includegraphics[width=.6\textwidth]{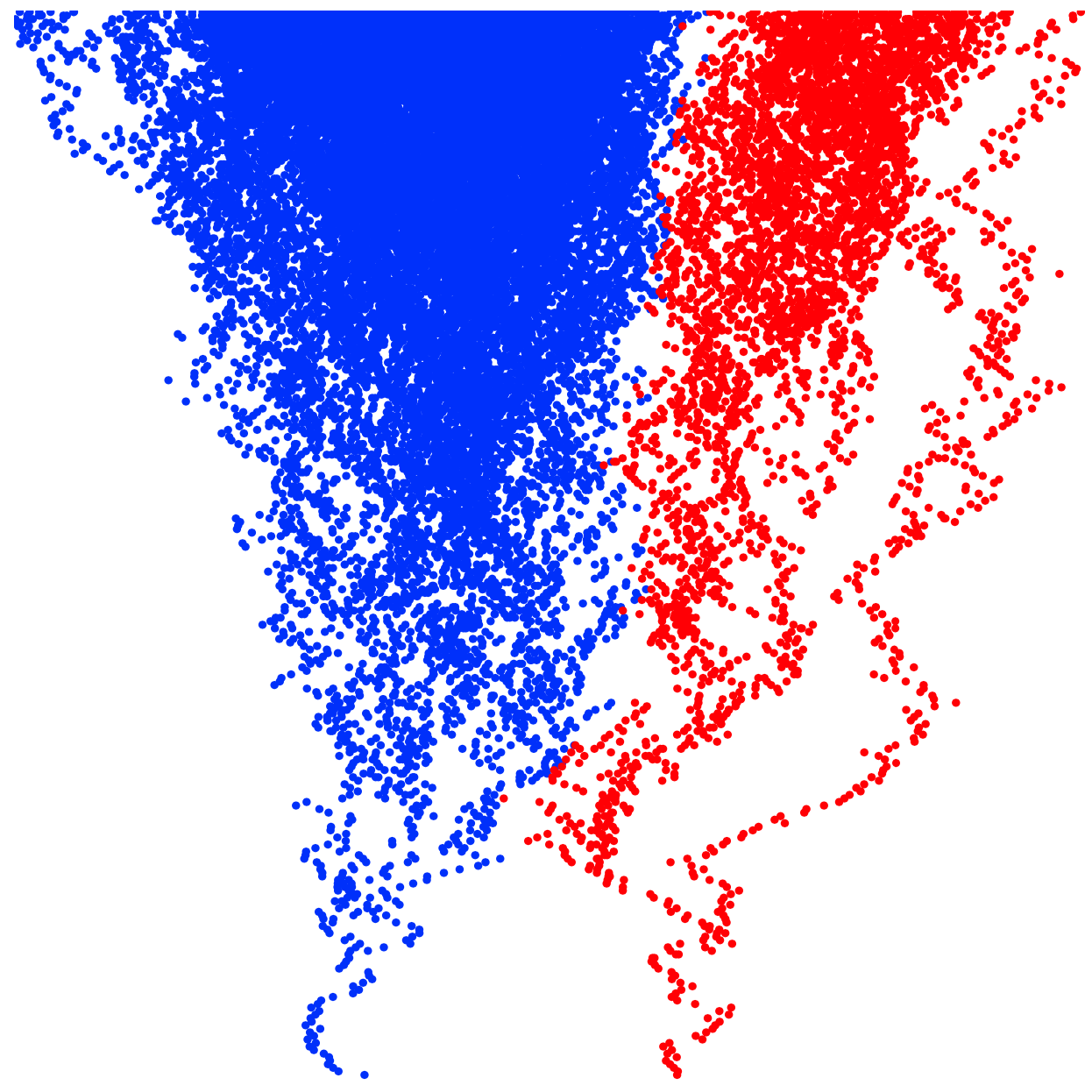}
\caption{An ABBM interface.}
\end{figure}

We remark that although we have defined the interface $I(t)$ 
as the midpoint of the interval $(I_-(t),I_+(t))$, this choice is irrelevant.
It follows from the proof that the two limits
\[
\lim_{t\to\infty}\frac{I_+(t)}{t} 
= \lim_{t\to\infty} \frac{I_-(t)}{t} 
= \lambda_*
\]
are equal almost surely.

Theorem \ref{thm:interface} is perhaps somewhat surprising: 
It seems that if the slope of the interface is far from 0, one type ought to dominate.
One heuristic explanation for why a linear interface with non-zero slope 
is plausible is as follows:
Suppose the interface follows some function $I(t)$.
If we observe only one type of particle, say positive, then we see BBM with absorption 
whenever a particle hits $I(t)$.
The connection between the types comes from the fact that the number of particles hitting $I(t)$ from the two sides up to any $t$ must be equal.
If we take $I(t)=at$ for an arbitrary slope $a$, it is easy to calculate the expected rate for particles hitting the interface from either side: This rate is $e^t$ (from the branching) times is the probability that Brownian motion first hits $I(t)$ at time $t$.
For any linear interface, these probabilities are the same from either side.
If $a>0$, then the total number of particles to the right of $I(t)$ is much smaller, but most of them will hit the interface, whereas the number of particles to the 
left of $I(t)$ grows faster, but most of them are near $0$ and stay quite far from the interface.
If we replace $at$ by a significantly non-linear function, then the probability of first hitting the interface at time $t$ is significantly different from the two sides, making such interfaces highly unlikely.
This heuristic can be turned into a proof of a functional limit 
theorem, 
though it does not seem to rule out that $I(t) = t\cdot a(t)$,  
with $a$ that changes very slowly over exponential time scales.

\subsection{Background}

The study of diffusive annihilating particles, without branching,  
was proposed in the physics literature \cite{OZ,TW} as a model 
for the inert chemical reaction
$A+B\to\emptyset$. 
Such systems were analyzed rigorously 
by Bramson and Lebowitz~\cite{BL1,BL2} 
in the early 1990s (cf.\ 
Cabezas, Rolla and Sidoravicius \cite{CRS18}).
Diffusive (multi-type)
annihilating systems that incorporate branching 
were studied more recently by 
Ahlberg, Griffiths, Janson and Morris~\cite{AGJM,AGJ}.
The current article is in part inspired by these works.

The question of coexistence in models for 
competing growth dates back to
H\"aggstr\"om and Pemantle~\cite{HP98}.
In this context, coexistence is closely 
linked to the existence of geodesics in first-passage percolation, 
as made precise by Hoffman~\cite{H08} and Ahlberg~\cite{A21}, 
and to the growth of arms in diffusion-limited aggregation (DLA), 
as established by Sidoravicius and Stauffer~\cite{SS19}.
Determining more general conditions for coexistence has 
posed a significant challenge, and remains elusive in 
several asymmetric spatial models for competing growth, 
see, e.g., \cite{HP00,B14,DJT20}.

A related direction is BBM with absorption.
If we imagine the interface to be given, 
then on either side of the interface we observe a simple BBM, 
and particles that hit the interface are annihilated.
The study of single-type BBM with absorption 
was initiated by Kesten~\cite{K}, and continued 
more recently by, e.g., 
Berestycki, Berestycki and Schweinsberg~\cite{BBS13,BBS14,BBS15} 
and others.
Our results complement these works.

We note that coexistence in two-type ABBM 
is consistent with \cite{K}, 
which suggests that particles of one type may 
escape towards $+\infty$ while being chased and annihilated 
by particles of the other type.
In contrast, in the 
discrete version of ABBM, 
the {\bf annihilating branching random walk} (ABRW), 
studied on finite connected graphs in \cite{AGJM},  
coexistence of two types is \emph{not} possible. 
This indicates that it is the unbounded nature of the 
real line that makes coexistence possible.
We expect that, as in \cite{AGJM}, 
coexistence is not possible for ABBM on a circle or a bounded interval. 
In contrast, our methods extend straightforwardly to show that 
ABRW on $\Z$ does in fact have positive probability 
of coexistence, but we omit the details. 

For three or more competing types the situation is more delicate. 
As shown in \cite{AGJM}, there are finite connected graphs 
for which coexistence in ABRW is possible.
On the other hand, Ahlberg and Fransson~\cite{AF} 
have shown that coexistence is not possible for any number 
of colors in ABRW on finite paths or cycles.

In this paper, we focus solely on the 1-dimensional setting. 
In a higher-dimensional setting one would need to define 
precisely under what conditions particles annihilate, 
since Brownian motions do not collide.
We believe that several reasonable models of 
annihilating Brownian motions in $\R^d$ 
do have a positive probability of coexistence.

We also believe that the question of coexistence 
does not depend on the precise branching mechanism.
For the sake of conciseness and clarity, 
we focus in this paper on the simple case of dyadic branching, 
where each particle branches into a pair of particles at the same location.
However, see Section \ref{sec:open} for a discussion of a 
more general context in which our methods should apply, 
and for further directions of study and open problems.

\subsection{Outline}

The main techniques used to derive the results of this paper 
are based on martingales and coupling arguments.
More precisely, we use additive martingales  
associated with BBM as a measure of activity 
in a narrow region of the real line.
By comparing the additive martingales associated with particles of different types, 
we show that some regions of $\R$ have more particles of one type 
(and other regions have more of the other types) implying coexistence.

To deduce the existence of a limiting speed of the interface, 
a more refined analysis is required in order to 
determine which regions of $\R$ have particles of each type.
A significant obstacle in this regard is to prove that certain 
martingale limits are almost surely unequal. 
Our argument for this is based on a more elaborate coupling, 
inspired by \cite{AGJM}.

Interestingly, to establish the existence of a limiting speed, and to rule out atoms of its distribution in $(-\sqrt{2},\sqrt{2})$, the additive martingales will suffice. However, at the critical speed $\pm\sqrt{2}$ the additive martingales provide no information, so in order to rule out a deadlock at the endpoints $\pm\sqrt{2}$, we will require the derivative martingales.

In Section \ref{sec:background} we recall the essentials of additive 
and derivative martingales martingales associated with dyadic, single-type BBM, which play a key role in our proofs.
In Section \ref{sec:coexistence} we construct a ``conservative'' coupling and prove Theorem \ref{thm:multi_coex}.
In Section \ref{sec:refined} we introduce  a refined coupling, which we use to obtain a more detailed description of the additive martingales associated with ABBM.
In Section \ref{sec:interface} this is used used to prove Theorem \ref{thm:interface}.
A number of open problems are discussed in Section \ref{sec:open}.

\subsection{Acknowledgments}

We thank Julien Berestycki for many useful discussions, 
and, in particular, for 
directing us to the work of Madaule~\cite{Mad16}, 
which allowed us to complete the proof of Proposition \ref{prop:no_sqrt2}. 

OA was supported in part by NSERC. 
This work was in part supported by the 
Swedish Research Council through grant 2021-03964 (DA).

\section{BBM and associated martingales}\label{sec:background}

\subsection{BBM counting measures}

At any given time, the state of BBM or ABBM 
is a finite configuration of particles on $\R$, 
formally a set $\{x_j\}_{j\in J}$ with some unspecified index set $J$.
More formally, we may represent a configuration of particles 
as a counting measure $\Xi=\sum_{j\in J}\delta_{x_j}$.
A configuration consisting of positive particles at $\{x_j\}_{j\in J}$ 
and negative particles at $\{y_j\}_{j\in J'}$ 
can similarly be represented by a pair of 
counting measures $(\Xi^-,\Xi^+)$, 
where $\Xi^-=\sum_{j\in J'}\delta_{y_j}$ 
and $\Xi^+=\sum_{j\in J}\delta_{x_j}$.
Since positive and negative particles never occupy 
the same location without annihilating, 
we can also consider the signed measure 
$\Xi^+-\Xi^-$ as a complete description of the state of an ABBM.
This notation extends naturally to configurations with 
more than two types, 
with a measure $\Xi^c$ representing particles of color $c$.

The notions of finite and non-trivial configurations 
defined above can more formally be expressed in terms 
of the counting measures. 
A configuration of (one or more types of) particles 
on $\R$ is said to be {\bf finite} if the corresponding 
counting measures are finite.
A configuration of two (or more) types is 
said to be {\bf non-trivial} if particles of different types 
occupy distinct locations, i.e., 
if the counting measures have disjoint support.

\subsection{BBM additive martingales}\label{sec:martingale}

The additive martingales
is one of the most fundamental objects associated with BBM.
Formally, let 
\[
\{X_j(t):j\in J(t),\,t\ge0\}
\] 
denote a BBM, 
which is equivalently represented as a measure-valued process
\begin{equation}\label{eq:counting}
\Xi_t:=\sum_{j\in J(t)}\delta_{X_j(t)}.
\end{equation}
The set $J(t)$ denotes an arbitrary index set 
for particles in the configuration at time $t$.
Usually (unless stated otherwise) we will assume that the process
starts with a single particle at the origin. 
We do not specify a canonical choice for labels.
Since particles are exchangeable, 
the specific choice of index set is irrelevant.

Let 
\begin{equation}\label{E_N}
N(t):=|J(t)|
\end{equation}
denote the number of particles at time $t$. 
Then $(N(t))_{t\ge0}$ evolves as a (single-type) 
continuous-time Markov branching process, 
in which individuals split in two at rate 1.
A straightforward calculation 
shows that $(e^{-t}N(t))_{t\ge0}$ is a non-negative martingale, 
so almost surely converges.

More generally, for $\lambda\in\R$ and $t\ge0$, we define 
\begin{equation}\label{E_Yt}
Y_\lambda(t) 
:= \sum_{j\in J(t)}e^{\lambda X_j(t)}  
= \int e^{\lambda x}\,d\Xi_t(x).
\end{equation}
By conditioning on the branching times, 
and recognizing $\E[e^{\lambda X_j(t)}]$ 
as the moment generating function of a 
centered Gaussian variable with variance $t$, 
we find that
\[
\E[Y_\lambda(t)] 
= \E[N(t)] \E[e^{\lambda X_j(t)}] 
= e^{\hat\lambda t},
\]
where
\begin{equation}\label{E_lam}
\hat\lambda 
:= 1+\frac{\lambda^2}{2}.
\end{equation}
Due to the Markov property of BBM, it follows that
\begin{equation}\label{E_Wlamt}
W_\lambda(t) 
:= e^{-\hat\lambda t}\,Y_\lambda(t)
\end{equation}
is a non-negative martingale.
Therefore, for any given $\lambda\in\R$, the limit 
\begin{equation}\label{E_Wlam}
W_\lambda 
:= \lim_{t\to\infty}W_\lambda(t)
\end{equation}
almost surely exists. 
For $\lambda\in\R$, 
the process $(W_\lambda(t))_{t\ge0}$ 
is referred to as the {\bf additive martingale} 
at $\lambda$ associated with the BBM $\{X_j(t)\}$.
We refer to a BBM, its additive martingales, 
and the limit of its additive martingales 
as \textbf{standard} if it is started with a single particle at the origin.

Certain changes to the configuration of a BBM 
have a simple effect on its long-term behavior 
and the additive martingales.
Specifically, 
translating the initial set results in translation of the entire process.
Taking a union of several configurations 
yields the union of the resulting processes.
For the additive martingales this leads to the following.

\begin{lemma}\label{lem:additivity}
Let $\{X'_j(t)\}$ be a BBM started with particles 
initially at positions $a_1,a_2,\dots,a_m\in\R$, 
and let $W'_\lambda$ denotes the limit of the 
associated additive martingale. 
Then
\begin{equation}\label{eq:mart_multiple}
W'_\lambda 
\stackrel{d}{=} 
\sum_{k=1}^me^{\lambda a_k}\,W_\lambda^{(k)},
\end{equation}
where $W_\lambda^{(1)},\dots,W_\lambda^{(m)}$ 
are IID standard additive martingale limits.
\end{lemma}

By \eqref{E_N} and \eqref{E_Yt}, 
we have that $W_\lambda(t)=N(t)$ when $\lambda=0$.
It follows from some versions of the Kesten--Stigum theorem 
that $W_0>0$.
For other $\lambda$, positivity of the limit $W_\lambda$ 
has been studied extensively.
This originates with the work of Biggins~\cite{biggins77} (cf.\  
\cite{neveu88,Big92,lyons97}).

\begin{prop}
  \label{prop:mart_lim}
  For $|\lambda|<\sqrt{2}$, the standard additive martingales converge in $L^1$, and $\P(W_\lambda>0)=1$.
  For $|\lambda|\ge\sqrt{2}$, we have $\P(W_\lambda=0)=1$.
\end{prop}

In addition to the above, let us also mention that $W_\lambda$ 
is known \cite{Big92} to be analytic in $\lambda\in[\sqrt{2},\sqrt{2}]$.
Proposition \ref{prop:mart_lim} can be used to establish a bound 
on the maximal displacement of particles
(which is usually proved by more elementary methods).
Let 
\[
M(t)
:=\max\{X_j(t):j\in J(t)\}
\] 
denote the location of the rightmost particle at time $t$.
It is well known that $M(t)\sim\sqrt{2} t$.
The upper bound for this follows 
using \eqref{E_Yt}--\eqref{E_Wlam} and Proposition \ref{prop:mart_lim} 
applied to $\lambda=\sqrt{2}$,  
which implies 
that, almost surely, 
\[
e^{\sqrt{2}M(t)-2t}\le W_{\sqrt{2}}(t)
\to 0,\quad\text{as }t\to\infty.
\]
Consequently, almost surely, we have $\sqrt{2}M(t)-2t\to-\infty$, 
and so
\begin{equation}\label{eq:max_bound}
\P(M(t) \le \sqrt{2}t \text{, for all large }t) = 1.
\end{equation}

\subsection{The BBM derivative martingale}
\label{S_DMG}

For $\lambda\in(-\sqrt{2},\sqrt{2})$ the additive martingale is non-zero, 
and will provide us with information regarding the evolution of particles. 
However, for $\lambda=\pm\sqrt{2}$ it is zero. 
For these values, finer information can be obtained from 
the so-called derivative martingale.

Since for every $\lambda$ we have a martingale 
$W_\lambda(t)$, it follows that the derivative 
$\frac{d}{d\lambda} W_\lambda(t)$ is also a martingale, 
known as the {\bf derivative martingale} at $\lambda$.
This has been used mainly with $\lambda=\pm\sqrt{2}$ 
to study the extremal values of BBM.
More specifically, define
\[D(t) := -\frac{d}{d\lambda} W_{\sqrt{2}}(t)
  = e^{-2t} \sum_{j\in J(t)} (\sqrt{2}t - X_j(t))e^{\sqrt{2} X_j(t)}. \]

Lalley and Sellke \cite{LS87} proved that 
$D_\infty := \lim_{t\to\infty} D(t)$ exists almost surely and in $L^1$, 
and that $D_\infty>0$ almost surely.
While derivatives and limits cannot in general be exchanged, 
the limit of the derivative martingale is related to the derivative of the limit $W_\lambda$.
For $|\lambda|<\sqrt{2}$, the limit and derivative can be exchanged since $W_\lambda(t)$ converges as an analytic function \cite{Big92}.
At $\lambda=\sqrt2$ the limit is not analytic.
Indeed the right derivative is $0$ since $W_\lambda=0$ for $\lambda>\sqrt{2}$.
However, the left derivative is related to the limit $D_\infty$, 
as the following theorem of Madaule~\cite{Mad16} shows.

\begin{thm}[{Madaule \cite[Theorem 1.1]{Mad16}}] \label{thm:madaule}
  Almost surely we have
  \[ D_\infty = \frac12 \lim_{\lambda\uparrow\sqrt{2}} \frac{W_\lambda}{\sqrt2-\sqrt{\lambda}}. \]
\end{thm}

We remark that this result is stated
in \cite{Mad16}
 for more general branching random walks, and not specifically BBM.
However, it can be applied to BBM by considering BBM at integer times. 
Indeed, it is easy to verify that the conditions in \cite{Mad16} hold.
(Madaule also uses a normalization for which the critical $\lambda$ is $1$, 
so one would need to rescale space or time for BBM by $\sqrt{2}$.)

We will also require a second fact, relating 
the derivative martingale to the displacement of the maximal particle.
There are several results along these lines.
The one which we will use is due to Arguin, Bovier and Kistler \cite{ABK13}.
Let 
\[
Q(t) := M(t) - \sqrt{2}t + \frac{3}{2\sqrt2} \log t,
\] 
and for $x\in\R$ let
\[
  F_T(x) := \frac{1}{T}\int_0^T {\bf 1}_{\{Q(t)\leq x\}} dt.
\]
be the empirical distribution function of $Q(t)$ up to time $T$. 
The limit of $F_T$ is the distribution function of a Gumbel 
random variable, shifted by $\log D_\infty$. 

\begin{thm}[{Arguin et al.\ \cite[Theorem 1.1]{ABK13}}] \label{T_ABK}
  Let $D_\infty$ be the limit of the derivative martingale and $F_T(\cdot)$ be as above.
  Then almost surely
  \[
    \lim_{T\to\infty} F_T(x) = \exp\left(-C D_\infty e^{-\sqrt{2} x} \right),
  \]
  where $C$ is some constant.
\end{thm}

In particular, the fraction of time that 
\[
M(t) \leq \sqrt{2}t - \frac{3}{2\sqrt2}\log t
\] 
converges to $\exp(-C D_\infty)$.

\subsection{A witness of activity}

A key observation that will be central in our arguments 
is that the main contribution to $W_\lambda(t)$ 
comes from particles located in a narrow interval around $\lambda t$.
Although this is quite natural, it (to the best of our knowledge) 
does not appear explicitly in the literature.
There are analogues showing that the dominant contribution 
to the derivative martingale $D(t)$ comes from particles 
near the maximum $\sqrt{2}$, and Theorem \ref{T_ABK} can be viewed as one such result.

In order to make this precise, let $\eps>0$ be arbitrary but fixed,
and consider particles within distance 
$\alpha_t=t^{1/2+\eps}$ of $\lambda t$.
Let 
\begin{equation}\label{eq:S_set}
S_\lambda(t):=\{x\in\R:|x-\lambda t|<\alpha_t\},
\end{equation}
and define (cf.\ \eqref{E_Wlamt})
\begin{equation}\label{E_barWlamt}
\overline W_\lambda(t)
:= e^{-\hat\lambda t}
\sum_{X_j(t)\in S_\lambda(t)}e^{\lambda X_j(t)}.
\end{equation}
In other words, we obtain $\overline W_\lambda(t)$ 
by including only particles within distance 
$\alpha_t$ of $\lambda t$ in 
the definition \eqref{E_Yt} of $Y_\lambda(t)$.
We note that $W_\lambda(t)-\overline W_\lambda(t) \ge 0$, 
and that the process 
$(\overline W_\lambda(t))_{t\ge0}$ is {\it not} a martingale, 
since changes from particles entering 
and exiting $S_\lambda(t)$ have non-zero mean.

The remainder of this section is 
devoted to showing that the difference between 
$W_\lambda(t)$ and $\overline W_\lambda(t)$ 
vanishes as $t\to\infty$, so that their limits coincide.

\begin{prop}\label{prop:mart_con}
  For every $\lambda\in\R$, almost surely, we have 
  $W_\lambda(t)-\overline W_\lambda(t) \to 0$ as $t\to\infty$.
  In particular, almost surely, the limits are equal (cf.\ \eqref{E_Wlam}), i.e., 
  \begin{equation}\label{E_barWlam}
    \overline W_\lambda := \lim_{t\to\infty}\overline W_\lambda(t)
  \end{equation}
  exists, and $W_\lambda=\overline W_\lambda$. 
\end{prop}

Combining Propositions \ref{prop:mart_lim} and \ref{prop:mart_con} 
we obtain that for $\lambda\in(-\sqrt{2},\sqrt{2})$ 
the limit $\lim_{t\to\infty}\overline W_\lambda(t)$ 
exists and is strictly positive with probability one. 
Since $\overline W_\lambda(t)$ 
only takes into account particles 
within distance $\alpha_t$ of $\lambda t$, 
this implies the {\it presence} 
of particles within distance $\alpha_t$ of $\lambda t$, for all large $t$. 
More formally, for every $\lambda\in(-\sqrt{2},\sqrt{2})$, we have that 
\begin{equation}\label{eq:part_existence}
\P\big(\exists t_0,\, \forall t\ge t_0,\, \exists j\in J(t),\, X_j(t)\in S_\lambda(t)\big)
=1.
\end{equation}
This property will be our main application of the additive martingale. 
We note that \eqref{eq:part_existence} 
together with \eqref{eq:max_bound} 
implies the matching (with \eqref{eq:max_bound} above) 
lower bound for $M(t)$, 
yielding that almost surely 
\begin{equation}\label{eq:max_lim}
\lim_{t\to\infty}\frac{M(t)}{t}=\sqrt{2}.
\end{equation}

Although intuitive, we could not find a reference
for Proposition \ref{prop:mart_con} in the literature. 
As such, we will give a detailed proof, 
based on large deviations of Gaussian variables.
Related results have previously 
been obtained by Biggins~\cite[Corollary~4]{Big92} 
by analytic methods.
The idea of the proof is that at any fixed time 
$W_\lambda(t)-\overline W_\lambda(t)$ 
can be bounded by its expectation.
This will give convergence along some subsequence.
To interpolate between these times we use a bound on 
the maximal displacement of a particle during a fixed length interval, 
implying that $W_\lambda(t)-\overline W_\lambda(t)$ 
does not increase very much during this interval.

Before stating the lemma, 
we recall that the tail of a standard Gaussian satisfies
\begin{equation}\label{eq:Gauss_tail}
1-\Phi(x)
\sim
\frac1x\frac{1}{\sqrt{2\pi}}e^{-x^2/2},
\quad\text{as }x\to\infty.
\end{equation}
In particular, 
$1-\Phi(x)\le e^{-x^2/2}$ for all large $x>0$.

\begin{lemma}\label{lem:speedy}
  A particle in BBM at time $t'$ is \textbf{speedy} in the 
  time interval $[t,t']$ if at some point during $[t,t']$ 
  it reached distance 1 from its position at time $t$.  
  Then, for $\delta$ small enough, 
  the probability that there are any speedy particles 
  in $[t,t+\delta]$ is at most $5e^{t-1/(2\delta)}$.
\end{lemma}

\begin{proof}
We estimate the expected number of speedy particles in $[t,t+\delta]$.
The expected number of particles existing at time 
$t+\delta$ is $e^{t+\delta}$.
For each of these, its trajectory over $[t,t+\delta]$ 
is a simple Brownian motion with some starting point.
By the reflection principle, 
the probability that it reaches distance 1 from its starting point 
is at most $4(1-\Phi(1/\sqrt{\delta}))$, 
which for $\delta$ small enough is at most $4e^{-1/(2\delta)}$.
Thus the probability that there exist speedy particles 
is at most $4e^{t+\delta-1/(2\delta)}$,
which implies the claim for $\delta$ small enough.
\end{proof}

\begin{proof}[Proof of Proposition \ref{prop:mart_con}]
Fix a sequence of times $t_n=n^{1/3}$ 
(the exponent $1/3$ is somewhat arbitrary 
but should be less than $1/2$).
To keep calculations clear, we define two processes: 
the process 
\[
Z_t 
= W_\lambda(t)-\overline W_\lambda(t) 
= \sum_{j}e^{\lambda X_j(t)} {\bf 1}_{\{|X_j(t)-\lambda t|>\alpha_t\}}
\]
that we wish to show tends to 0, and a second process 
\[
Z'_t 
= W_\lambda(t)-\overline W_\lambda(t) 
= \sum_{j}e^{\lambda X_j(t)} {\bf 1}_{\{|X_j(t)-\lambda t|>\alpha_t/2\}}
\]
counting particles outside a smaller interval of radius $\alpha_t/2$. 
  
First, we claim that 
\begin{align*}
\E [Z'_t]
&= e^{-\hat\lambda t}\,\E\bigg[\sum_{j:|X_j(t)-\lambda t| > \alpha_t/2} e^{\lambda X_j(t)}\bigg]\\
&=\int_{\{|x-\lambda t|>\alpha_t/2\}}\frac{1}{\sqrt{2\pi t}}e^{-(x-\lambda t)^2/2t}\,dx.
\end{align*}
Indeed, this follows noting that the 
expected number of particles is $e^t$, 
that each particle's position is distributed as a $N(0,t)$ random variable, 
and the value of $\hat\lambda$.
  
The integral corresponds to the probability 
that a $N(\lambda t,t)$ random variable deviates 
by at least $\alpha_t/2$ from its mean. Hence 
\[
\E [Z'_t] 
= 2 [1-\Phi(\alpha_t/2\sqrt{t})] 
\leq 2e^{-\alpha_t^2/8t}
\]
for large enough $t$. 
Therefore, with our choice of $t_n=n^{1/3}$ 
and $\alpha_t = t^{1/2+\eps}$, we have that
\[
\E [Z'_{t_n}] \leq 2e^{-n^{2\eps/3}/8}.
\]
Fix any $0 < \delta < 2\eps/3$.
By Markov's inequality and Borel--Cantelli, 
we almost surely have $Z'_{t_n} \leq e^{-n^\delta}$ 
for all large enough $n$.
  
Next, we wish to bound $Z_t$ 
for $t\in[t_{n},t_{n+1}]$ in terms of $Z'_{t_n}$.
By Lemma \ref{lem:speedy} the probability that any particle moves 
more than distance 1 in $[t_n,t_{n+1}]$ 
is at most $Ce^{n^{1/3}-cn^{2/3}}$, for some constants $c,C>0$.
By Borel--Cantelli, there are no such (speedy) particles 
in $[t_n,t_{n+1}]$ for all $n$ large enough.
We assume this is the case from here on.
Consider now any $s\in[t_n,t_{n+1}]$ and 
particle $X_j$ that contributes to $Z_s$.
This particle is not speedy in $[t_n,t_{n+1}]$.
Therefore, 
$e^{\lambda X_j(s)} \leq e^{|\lambda| + \lambda X_j(t_n)}$, 
where $X_j(t_n)$ is its location at time $t_n$ 
(or else of its ancestor at time $t_n$, if it was created after time $t_n$).
Moreover, since the particle is not speedy, 
it also contributed to $Z'_{t_n}$.
It follows that for $n$ large enough 
we have $Z_s \leq e^{|\lambda|} Z'_{t_n}$, 
and therefore $Z_s\to 0$ as $s\to\infty$.
\end{proof}

\section{Coexistence}\label{sec:coexistence}

In this section we address the question of coexistence in ABBM, 
and prove Theorem \ref{thm:multi_coex}.
For simplicity, we begin with the case of two types.

\begin{thm}\label{thm:coexistence}
For a two-type ABBM started from a finite, 
non-trivial initial configuration, 
there is a positive probability of coexistence. 
\end{thm}

A key ingredient of our argument is a 
``conservative'' construction of the two-type process, 
in which particles of different types ``merge'' instead of annihilating. 
This coupling was introduced by 
Ahlberg, Griffiths, Janson and Morris~\cite{AGJM} 
in the context of annihilating branching random walks.
One distinction in our context is that in~\cite{AGJM} 
the coupling was used to prove that coexistence is {\it not} possible, 
whereas we use it to prove that coexistence {\it is} possible.
The different behavior stems from the difference in the ambient space.
In~\cite{AGJM} the space is finite (a finite connected graph), 
whereas here it is infinite (the real line).

\subsection{A conservative coupling}
\label{ssec:coupling}

The construction involves a new type
of {\bf neutral} particles, 
which are created when positive and negative particles merge. 

Consider a system of positive, negative and neutral particles in which the 
particles perform independent BBMs
(with children retaining their parent's type).
Particles of equal types do not interact with each other when intersecting paths,
and neutral particles do not interact with particles of the other types.
When a positive and a negative particle collide (and annihilate), 
a neutral particle is created at the point of annihilation. 
In other words, a neutral particle is a pair of
positive and negative particles which 
have stuck together, and then continue to move 
and branch in unison thereafter, independently of everything else.

Consider the above system started from a finite, 
non-trivial configuration of positive
and negative particles, and no neutral particles.
Let $\{X^+_j(t)\}$, $\{X_j^-(t)\}$ and $\{X_j^\circ(t)\}$ 
denote the sets of positive, negative and neutral particles at time $t\ge0$.
Likewise, let $\Xi_t^+$, $\Xi_t^-$ and $\Xi_t^\circ$ 
denote the corresponding counting measures 
(as in~\eqref{eq:counting}).
A fundamental observation from \cite{AGJM} is that, 
by construction, the processes
$(\Xi_t^++\Xi_t^\circ)_{t\ge0}$ 
and $(\Xi_t^-+\Xi_t^\circ)_{t\ge0}$ 
are single-type (non-annihilating) BBMs.

Let $(W_\lambda^+(t))_{t\ge0}$ 
and $(W_\lambda^-(t))_{t\ge0}$ denote the additive martingales 
associated with $(\Xi_t^+ + \Xi_t^\circ)_{t\ge0}$ 
and $(\Xi_t^- + \Xi_t^\circ)_{t\ge0}$. 
Since these are additive martingales of single-type BBMs, 
by Proposition \ref{prop:mart_lim}, the limits
\begin{equation}\label{eq:conservative_limits}
  W_\lambda^\pm:=\lim_{t\to\infty}W_\lambda^\pm(t)
\end{equation}
almost surely exist, and are positive on $(-\sqrt2,\sqrt2)$.

From a comparison of the two martingales we will be 
able to demonstrate the existence of 
positive and negative particles in different regions of the real line.
Of course, the two processes, and hence the associated martingales, 
are closely dependent through the mutual inclusion of the neutral particles.
This poses little difficulty in establishing coexistence, 
but will have to be treated more carefully when analyzing the 
limit speed of the interface in Sections \ref{sec:refined} and \ref{sec:interface} below.

\subsection{Two-type coexistence}

We now show, using the martingales $W_\lambda^\pm(t)$, 
that in a two-type ABBM there is 
positive probability that the two types coexist forever. 

\begin{proof}[Proof of Theorem \ref{thm:coexistence}]
From any non-trivial finite initial condition, there is a 
positive probability that at some later time, there is a 
single pair of particles of opposite types at a given distance apart.
Hence, by the Markov property, translation invariance and symmetry, 
it suffices to establish a positive probability of coexistence 
from an initial configuration consisting of a 
single negative particle at $-a$ and a single positive particle at $+a$, 
for some arbitrarily large $a$.

Fix some $\lambda\in(0,\sqrt{2})$. 
By Lemma \ref{lem:additivity} the martingale limits 
$W_\lambda^\pm$ in~\eqref{eq:conservative_limits} satisfy
\[
W_\lambda^\pm
\stackrel{d}{=}
e^{\pm\lambda a}\,W_\lambda, 
\]
where $W_\lambda$ is a standard martingale limit.  
By Proposition \ref{prop:mart_lim}, 
the limit $W_\lambda$ is almost surely positive.
It follows that, for any large $a>a_0$, we have that 
\[
\P(W_\lambda^-<1<W_\lambda^+)\ge3/4.
\]
Moreover, on this event, for all large $t$, 
we have that
\[
W^+_\lambda(t) - W^-_\lambda(t) 
= e^{-\hat\lambda t} \bigg(\sum_i e^{\lambda X^+_i(t)} - \sum_j e^{\lambda X^-_j(t)} \bigg)
> 0.
\]
However, this can only happen if positive particles 
are present in the system for all large $t$, 
that is, if the positive particles survive.
(In fact, by Proposition \ref{prop:mart_con}, this 
more specifically demonstrates the 
existence of positive particles in the vicinity of $\lambda t$, 
which will be useful later on, for the interface speed.)

By symmetry, we also have
\[
\P(W^-_{-\lambda} > 1 > W^+_{-\lambda})\ge3/4. 
\]
On this event, the negative particles survive (in the vicinity of $-\lambda t$). 
Consequently, for large $a$, the probability of coexistence is at least $1/2$.
\end{proof}

\subsection{Multi-type coexistence}
\label{S_multicoex}

Next, we generalize to the case
of a $k$-type ABBM. 

\begin{proof}[Proof of Theorem \ref{thm:multi_coex}]
Consider ABBM with an initial (finite
and non-trivial) configuration of $k\ge2$ types.
Either $k=2\ell$ or $k=2\ell+1$, for some $\ell\ge1$.
For convenience, 
we label the $k$ types using the index set  
\[
I_k
=
\begin{cases}
\{\pm1,\ldots,\pm\ell\}&k=2\ell\\
\{0,\pm1,\ldots,\pm\ell\}&k=2\ell+1.
\end{cases}
\]

Let $\Xi^i(t)$ be the counting measures 
for particles of type $i$ at time $t$.  
We describe an extended coupling similar 
to the one used in the proof of Theorem \ref{thm:coexistence} above.
Instead of a single type of neutral particle, as before, 
we now have ${k\choose 2}$ types of neutral particles, 
which we denote by the sets $\{i,j\}$ with $i\neq j\in I_k$.
When particles of types $i\neq j$ collide, 
instead of annihilating as usual, 
they combine to form a neutral particle of type $\{i,j\}$.  
These new particles are neutral in the sense that they 
do not interact with any other particles, 
however, they do continue to perform BBM, 
with neutral offspring of the same type $\{i,j\}$.  
  
Let $\Xi^{\{i,j\}}_t$ be the counting measures for 
neutral particles of type $\{i, j\}$ at time $t\ge0$.
Observe that, for any $i\in I_k$, 
\begin{equation}\label{E_hatXi}
\tilde  \Xi^i_t 
:= \Xi^i_t + \sum_{j:j\neq i} \Xi^{\{i,j\}}_t
\end{equation}
is the counting measure of a single-type BBM.
Let 
\[
W^i_\lambda(t) 
:= e^{-\hat\lambda t} \int e^{\lambda x} d\tilde \Xi^i_t(x)
\]
denote the corresponding additive martingale, 
and $\tilde W^i_\lambda$ its almost sure limit.
  
We will use these martingales to prove that 
coexistence occurs with positive probability, 
from a carefully constructed, 
``bell-shaped'' initial configuration, as in Figure \ref{F_ktype}.  
The key observation is that, by \eqref{E_hatXi}, 
we have for $t\ge0$ that
\[
\Xi^i_t
\geq 
\tilde\Xi^i_t - \sum_{j:j\neq i} \tilde\Xi^j_t, 
\]
since clearly each $\Xi^{\{i,j\}}_t\le \tilde\Xi^j_t$. 
Therefore, on the event that for every $i$ 
there is some $\lambda_i$ so that
\begin{equation}\label{E_Wi}
W^i_{\lambda_i} > \sum_{j:j\neq i}  W^j_{\lambda_i}, 
\end{equation}
it follows that all types survive. 
  
\begin{figure}[h!]
\centering
\includegraphics[scale=0.375]{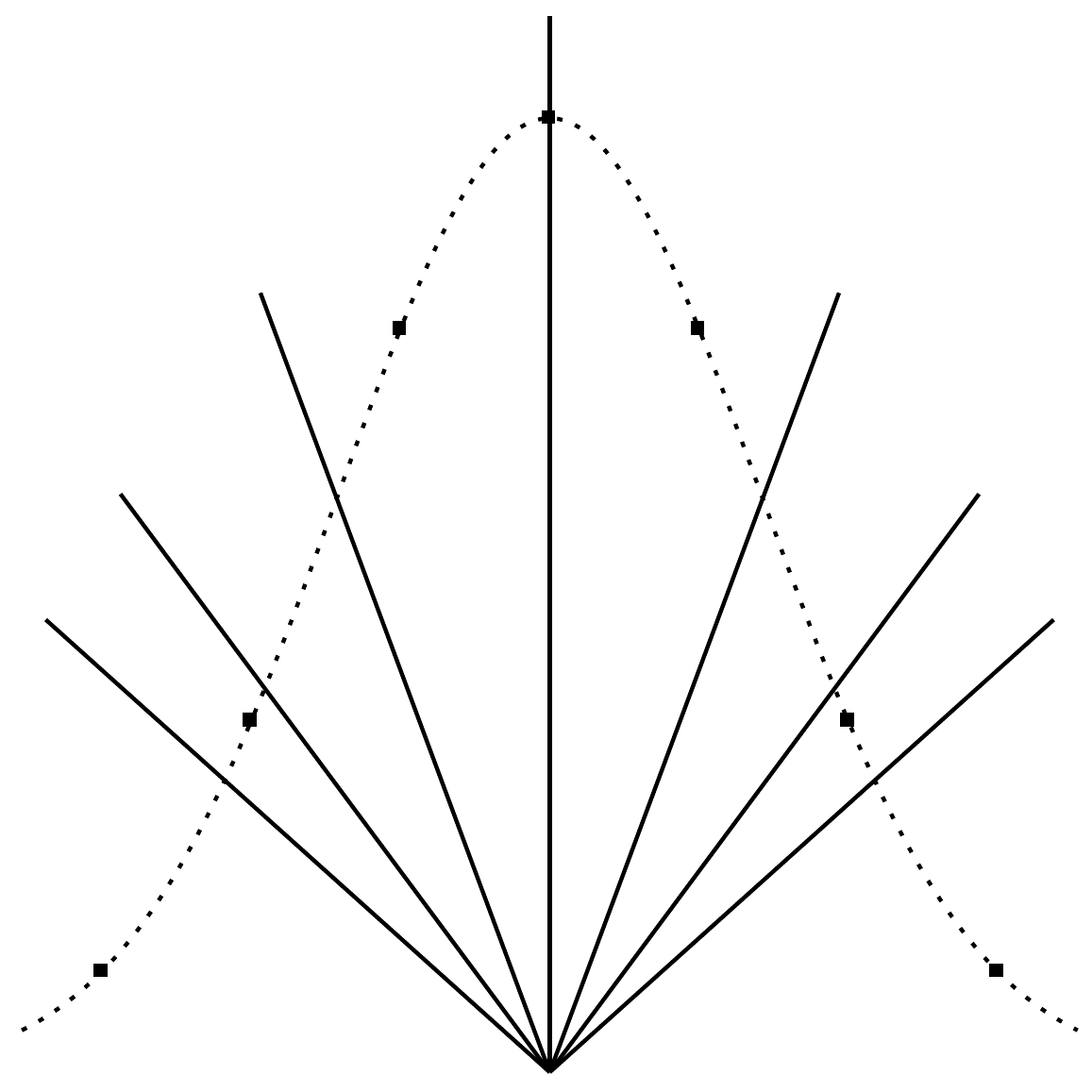}
\caption{The number of particles $n_i$ at positions $a_i$
are chosen along a bell-shaped curve in such a way that, 
with positive probability, each type $i$ survives in 
the vicinity of $\lambda_i t$.
In this figure, $k=7$ and $\delta=1$. 
}
\label{F_ktype}
\end{figure}

We will now describe an initial configuration 
for which this holds with positive probability for the choice
\[
\lambda_i = \frac{\sqrt{2}i}{\ell+1}.
\]
Fix a small $\eps>0$ and a large $\delta\gg 1$.
Note that, permuting labels if necessary, 
there is a positive probability that the system 
will at some point transition to a configuration 
consisting of a total of $n=\sum_{i=1}^k n_i$ particles, 
with
\[
n_i=\exp\left(\frac{\delta(\ell^2-i^2)}{\ell+1}\right)
\]
particles of type $i$, all within distance $\eps$ of the point 
\[
a_i=\sqrt{2}i\delta.
\]
We may assume that $\delta=(\ell+1)\log m$ 
for some positive integer $m$ so that the $n_i$ are integers.

These parameters have been chosen so that, 
for all $i\neq j\in I_k$, 
\[
\log \frac{n_ie^{\lambda_i a_i}}{n_je^{\lambda_i a_j}}
=\frac{\delta (i-j)^2}{\ell+1}\ge\frac{\delta}{\ell+1}.
\]
Hence, for any $i\in I_k$, as $\delta\to\infty$, 
\begin{equation}\label{E_preWi}
n_i e^{\lambda_i a_i}\gg \sum_{j:j\neq i} n_j e^{\lambda_i a_j}.
\end{equation}

By Lemma \ref{lem:additivity}, each martingale limit $ W_\lambda^i$ 
is equal in distribution to the sum of $n_i$ 
independent standard martingale limits, 
distributed as $e^{\lambda x}\,W_\lambda$, 
where $x$ with $|x-a_i|<\eps$ 
is the position of each respective particle.
These differ by at most a bounded factor of 
$e^{(a_i\pm \eps)\lambda}$ from $W_\lambda$.
Therefore, by the law of large numbers for IID copies 
of $W_{\lambda_j}$, we may find $c>0$ such that, 
for each pair of $i,j\in I_k$ and all large $\delta$, 
we have
\[
\P\Big(cn_ie^{\lambda_j a_i}\le 
W_{\lambda_j}^i\le \frac{1}{c}n_ie^{\lambda_ja_i}\Big)>1-\frac{1}{2k^2}.
\]
Therefore, for a fixed $i\in I_k$, 
we obtain together with~\eqref{E_preWi} that
\[
\P\bigg(  W^i_{\lambda_i} > \sum_{j:j\neq i}  W^j_{\lambda_i}\bigg)
\ge
1-\frac{1}{2k}, 
 \]
for all large $\delta$. Consequently, for all large $\delta$, 
this holds simultaneously for all $i\in I_k$ with probability at least $1/2$.
On this event, by \eqref{E_Wi}, 
each type $i$ survives in the vicinity of $\lambda_i t$. 
\end{proof}

\section{Non-equality of martingale limits}\label{sec:refined}

Having established the possibility of coexistence, 
our next goal is to examine the asymptotic behavior 
of the interface in a two-type ABBM.
In proving Theorem \ref{thm:interface},
in Section \ref{sec:interface} below, 
our main tool once again 
will be the additive martingales 
used in the conservative coupling in Section \ref{sec:coexistence}. 
However, in order to obtain such a result, 
the conservative coupling 
will need to be refined,
and this is the topic of the current section. 

Note that, by Proposition \ref{prop:mart_con}, 
an inequality of martingale limits 
$W_\lambda^+ > W_\lambda^-$ implies existence 
of positive particles with asymptotic speed 
$\lambda$ (while the opposite inequality implies 
existence of negative particles with asymptotic speed $\lambda$).
As such, an important step towards Theorem \ref{thm:interface} 
will be to prove that, for almost every value of $\lambda$, 
the martingale limits are unequal.
While the martingale limits are known to be continuous, 
the limits $W^+_\lambda$ and $W^-_\lambda$ are not independent, 
so this inequality requires justification.

An analogous statement holds also for the derivative martingales. 
Let $D_\infty^+$ and $D_\infty^-$ denote the limits of the derivative 
martingales associated with the BBMs $(\Xi_t^++\Xi_t^\circ)_{t\ge0}$ 
and $(\Xi_t^-+\Xi_t^\circ)_{t\ge0}$.
To state the result, let $\mathcal{S}$ denote the event of {\bf survival}, 
i.e., that for all large times there are either positive 
or negative particles present in the system.
Note that $\mathcal{C}\subseteq\mathcal{S}$, 
since coexistence $\mathcal{C}$
requires the survival of both types
of particles. 

\begin{prop}\label{prop:equality}
  Consider the conservative coupling of ABBM with a finite, 
  nontrivial initial configuration of particles.
  For every $\lambda\in(-\sqrt{2},\sqrt{2})$,
  \[
    \P\big(\{W_\lambda^+=W_\lambda^-\}\cap \mathcal{S}\big) = 0.
  \]
  Moreover, the same holds for the derivative martingales:
  \[
    \P\big(\{D_\infty^+=D_\infty^-\}\cap \mathcal{S}\big) = 0.
  \]  
\end{prop}

The key additional ingredient in the proof of this proposition is a second, 
``enhanced'' coupling, inspired by~\cite{AGJM}. 
In this coupling, there are two closely related copies of ABBM,
constructed in such a way that
$\{W_\lambda^+=W_\lambda^-\}\cap \mathcal{S}$ 
cannot occur in both instances.
This allows us to bound the probability of this event away from $1$, 
and then the result follows by L\'evy's 0--1 law.

\subsection{An enhanced coupling}\label{ssec:refined}

As discussed above, we will require a refinement 
of the conservative coupling introduced in Section \ref{sec:coexistence}.
The key idea is to consider two copies of ABBM,  
which are initially identical and evolve together.
However, once the first branching event occurs, 
it is suppressed in one process, and allowed in the other.
This creates a discrepancy between the two systems, 
which we track by marking certain particles.
Instead of working with two ABBMs, the coupling uses a single ABBM, 
with the addition of marked particles 
denoting the difference between the two processes.
We proceed with formalizing this 
construction below.

The enhanced coupling is realized by a system 
of positive, negative and neutral particles, 
where each particle 
is also either {\bf marked} or {\bf unmarked}.
(Marked/unmarked particles give birth
to marked/unmarked particles.)
By the construction below, 
all marked particles  
will be either positive or negative; 
which of the two is determined by chance.
Moreover, marked particles survive forever, 
and evolve as a single-type BBM (without annihilation).

As before, all particles perform BBM,  
with unmarked particles interacting as usual, 
with positive and negative particles forming 
neutral particles upon meeting.
(In fact, \emph{unmarked} neutral particles 
can at this point be ignored, 
but \emph{marked} neutral particles will matter.)
Marked particles do not interact with each other, 
since marked particles of both signs 
will never be present at the same time.
On the other hand, the way in which 
marked particles interact with unmarked particles 
will depend on whether the first marked particle 
that appears in the system is positive or negative. 

{\bf Case 1 {\rm (positive marked particles)}.}
Suppose that marked positive and neutral, 
but no marked negative, 
particles are present in the system at some point. 
The marked particles interact with other
particles as follows. 
\begin{enumerate}
\item[(i)] Marked positive particles do not interact 
with unmarked positive/neutral particles. 
A marked positive particle and an unmarked negative particle 
form a marked neutral particle.
\item[(ii)] Marked neutral particles do not interact 
with unmarked negative/neutral particles. 
A marked neutral particle and an unmarked positive particle 
form an unmarked neutral particle and a marked positive particle.
\end{enumerate}

The last interaction can be thought of as the mark 
being transferred from the neutral to the unmarked positive particle.
Alternatively, if we think of a marked neutral 
particle as a marked positive and unmarked negative pair stuck together, 
then the last interaction may be seen as the unmarked positive 
particle taking the place of the marked positive in the pair, 
and the marked positive is set free.

{\bf Case 2 {\rm (negative marked particles)}.}
Suppose instead that marked negative and neutral, 
but no marked positive, particles appear at some point in the system.
The interactions between marked and 
unmarked particles are the same as Case 1, 
with positive/negative switching roles.

As we shall no longer care about unmarked neutral particles, 
any such particles that appear as a result of the above interactions 
can thereafter be ignored in both cases.

As noted, this construction  
can be seen as two coupled ABBMs.
One is obtained by simply ignoring all marks, 
and the other by deleting all marked positive 
(assuming Case 1, as Case 2 is symmetric) particles, 
thereby also 
converting marked neutral particles into negative particles.
It is easy to verify that each of these two projections map 
the refined dynamics described above to the usual ABBM dynamics.

As before we shall denote the counting measures associated 
to the configurations of unmarked positive 
and unmarked negative particles present at time $t$ 
by $\Xi_t^+$ and $\Xi_t^-$, respectively.
(Unmarked neutral particles can be ignored.)
In addition, we denote the counting measures 
associated with the configurations of marked positive, 
negative and neutral particles present at time $t$ 
by $\hat\Xi_t^+$, $\hat\Xi_t^-$ and $\hat\Xi_t^\circ$, respectively.
The construction will be such that only one of marked 
positive or marked negative will ever occur, 
so that one of $\hat\Xi_t^+$ and $\hat\Xi_t^-$ (possibly both) 
will be zero for {\it all} $t\ge0$.

In order to use the refined coupling, 
we need to also create a first marked particle, 
which will be the ancestor of all subsequent marked particles.
To this end, we will modify the behavior of the process at a 
random time $\tau_1$, as described below.
We proceed and construct a system of marked and unmarked 
particles evolving from a finite initial configuration, 
by describing its evolution for a random amount of time, 
after which we let the system evolve on its own as described above.
Let $(\Xi_0^+,\Xi_0^-)$ denote any finite, 
non-trivial configuration of positive and negative particles.
All particles are initially unmarked, and then
(if ever) marked particles will appear in the system 
after a certain random time.
We will refer to particles that are either unmarked positive, 
unmarked negative or marked neutral as {\bf active}.
For $t\ge0$, let $\mu(t)$ denote the number of 
active particles present at time $t$.

As time starts, let the particles present evolve 
according to independent Brownian motion, 
without branching but with annihilation 
(of unmarked positive and negative particles).
Let $\tau_1$ denote the first arrival time in 
an inhomogeneous Poisson process on 
$[0,\infty)$ with intensity measure $\mu(t)$. 
Note that $\mu(t)$ can decrease through annihilation, 
but cannot yet increase.
Note also that $\tau_1$ has the law of the time of the first branching event.
At time $\tau_1$, choose an active particle uniformly at 
random and position a marked particle of the same sign at the same location.
(Before time $\tau_1$ there are no marked particles present, 
so the particle chosen is unmarked.)
Note that $\tau_1=\infty$ in the case that all particles die out before $\tau_1$.
In this case the system is empty apart from unmarked neutral particles, 
and so the construction effectively terminates. 
Note also that if we ignore the marks, 
the resulting system is the same as the regular ABBM up 
until time $\tau_1$. 
After time $\tau_1$, the system evolves 
according to the dynamics of the refined coupling described above.

\subsection{Properties of the enhanced coupling}

The key claim stated below is that the coupling above 
is a coupling of two versions of the ABBM, 
where one has its first branching event suppressed.
In addition to this, we are also 
interested in the evolution of marked particles, 
as these particles are the difference between the two versions.

Let $A^+$ (resp.\ $A^-$) denote the event that $\tau_1$ 
is finite and that the particle chosen at time 
$\tau_1$ is positive (resp.\ negative).
The processes we shall consider are the following:
\begin{align*}
\Xi'&:=\big(\Xi_t^++\hat\Xi_t^+,\Xi_t^-+\hat\Xi_t^-\big)_{t\ge0},\\
\Xi''&:=\big(\Xi_t^++\hat\Xi_t^\circ\cdot{\bf 1}_{A^-},\Xi_t^-+\hat\Xi_t^\circ\cdot{\bf 1}_{A^+}\big)_{t\ge0},\\
\hat\Xi&:=\big(\hat\Xi_t^++\hat\Xi_t^-+\hat\Xi_t^\circ\big)_{t\ge0}.
\end{align*}
Their distributional properties are summarized in the following lemma, 
and follow immediately from the construction. 

\begin{lemma}\label{lma:processes}
Consider the enhanced coupling started 
from a non-trivial, finite configuration.
\begin{enumerate}
\item[(a)] The law of $\Xi'$ 
equals that of an ABBM started from $(\Xi_0^-,\Xi_0^+)$.
\item[(b)] The law of $\Xi''$ 
equals that of an ABBM started from $(\Xi_0^-,\Xi_0^+)$, 
which has its first branching event, 
at time $\tau_1$, suppressed.
\item[(c)] Conditioned on $\tau_1<\infty$, 
the law of the restriction of $\hat\Xi$ from time $\tau_1$ onwards 
equals that of a BBM 
started from $\hat\Xi_{\tau_1}^++\hat\Xi_{\tau_1}^-$.
\end{enumerate}
\end{lemma}

From Lemma \ref{lma:processes}, 
we see that the laws of 
$\Xi'$ and $\Xi''$ are not equal.
However, they
do not differ by much, since 
it is only a single branching event suppressed at a random time 
$\tau_1$ in $\Xi''$ that causes the difference. 
This observation is key to the following lemma.

\begin{lemma}\label{lma:distance}
For any initial configuration of particles, 
the total variation distance between the laws of $\Xi'$ and $\Xi''$ 
is bounded by $e^{-1}$.
\end{lemma}

\begin{proof}
By Lemma \ref{lma:processes}, the laws of the two processes 
$\Xi'$ and $\Xi''$ correspond to those of ABBMs 
in which the first branching event is allowed
in $\Xi'$ and suppressed in $\Xi''$.
The time of the first branching event in the first system is $\tau_1$, 
which is the first arrival in a Poisson process with intensity $\mu(t)$.
Recall that $\mu(t)$ is the number of 
active (unmarked positive, unmarked negative, and marked neutral) 
particles present at time $t$. 
Marked positive particles are descendants 
of the first particle to branch at time $\tau_1$ (before then there are no marked 
particles). These particles are suppressed 
in the second system, and not counted by $\mu(t)$. 
Thus $\mu(t)$ equals the total branching
rate in the second system. 
Let $\tau_2$ be the next arrival time, after $\tau_1$, 
in a Poisson process with intensity $\mu(t)$.
Then the first branching time 
in the second, suppressed process is given by $\tau_2$. 

It follows that the total variation distance 
between the laws of the processes $\Xi'$ and $\Xi''$ 
is bounded by the total variation distance 
between the laws of $\tau_1$ and $\tau_2$.
Formally, we may construct (yet another) coupling 
$(\bar\Xi',\bar\Xi'')$ of the two processes $\Xi'$ and $\Xi''$ 
as follows:
Let $\delta$ denote the total variation distance 
between the random variables $\tau_1$ and $\tau_2$, 
and let $(\bar\tau_1,\bar\tau_2)$ be a coupling of $\tau_1$ and $\tau_2$ 
such that $\bar\tau_1=\bar\tau_2$ with probability $1-\delta$.
Let $\bar\Xi'$ evolve so that its first branching event 
occurs at time $\bar\tau_1$.
On the event that $\bar\tau_1=\bar\tau_2$ set $\bar\Xi''=\bar\Xi'$, 
and on the event that $\bar\tau_1\neq\bar\tau_2$ let $\bar\Xi''$ 
evolve independently of $\bar\Xi'$ with its first branching event 
occurring at time $\bar\tau_2$.
Then on the event that $\bar\tau_1=\bar\tau_2$ 
we have $\bar\Xi'=\bar\Xi''$ at all times.
Note that it is possible that either
$\tau_1$, $\tau_2$ or both are $\infty$, but this
 does not pose a problem.
  
Therefore, it remains to show that
\begin{equation}\label{eq:TV}
d_{\rm TV}(\tau_1,\tau_2) \leq e^{-1}.
\end{equation}
The random variables $\tau_1$ and $\tau_2$ 
are the first and second arrival times of an inhomogeneous 
Poisson process on $[0,\infty)$ with (random) intensity measure $\mu(t)$.
Let $(N(t))_{t\ge0}$ denote the corresponding Poisson process.
We show that~\eqref{eq:TV} is true uniformly 
when conditioning on $\mu(t)$, so also on average.
  
Given $\mu(t)$ we have that $N(t)$ is a 
time change of a homogeneous Poisson process, 
possibly stopped at some finite time.
Specifically, let $\phi(t) = \int_0^t\mu(s)\,ds$.
Then $N(\phi^{-1}(t))$ is a standard Poisson process, 
with intensity 1 on $[0,\phi(\infty))$.
Thus it suffices to consider a homogeneous 
Poisson process on $[0,a)$ with $a$ possibly $\infty$, 
and bound the total variation distance between the first and second arrivals, 
again denoted $\tau_1,\tau_2$.
If $a=\infty$, then $\tau_1$ has density $e^{-x}$ 
and $\tau_2$ has density $xe^{-x}$ on $\R_+$.
If $a<\infty$, then the same laws are truncated at $a$ 
with the remaining probability being an atom at $\infty$.
The ratio between the conditional densities of 
$\tau_1$ and $\tau_2$, given $(\mu(s))_{s\ge0}$, 
is a decreasing function, 
so the set $A$ maximizing $\P(\tau_2\in A)-\P(\tau_1\in A)$ is 
of the form $[c,\infty]$, for some $c$.
We have 
\begin{align*}
\P(\tau_1 \ge t) &= \P(N(t)=0) = e^{-t} \\
\P(\tau_2 \ge t) &= \P(N(t)\le 1) = (1+t)e^{-t}.
\end{align*}
It follows that the total variation distance 
$\delta$ between $\tau_1$ and $\tau_2$ can be expressed as
\[
\sup_{t\ge0}\Big(\P(\tau_2\ge t)-\P(\tau_1\ge t)\Big)
= \sup_{t\ge0}\; te^{-t} = e^{-1},
\]
as required. (If part of the measure is transferred to $\infty$, 
the total variation distance can only become smaller.)
\end{proof}

\subsection{Proof of Proposition \ref{prop:equality}}

Note that Proposition \ref{prop:equality} contains a statement 
regarding the martingale differences 
$(W_\lambda^+(t)-W_\lambda^-(t))_{t\ge0}$ 
of the conservative coupling from Section \ref{ssec:coupling}.
However, in order to prove this statement we shall work 
with the enhanced coupling construction of Section \ref{ssec:refined}.
As such, the first step is to identify the analogue of the 
martingale differences in this enhanced  construction.

Fix $\lambda\in(-\sqrt{2},\sqrt{2})$. 
Let $\chi$ be the random variable that indicates the 
sign of the active particle drawn at time $\tau_1$. 
That is, $\chi=+1$ if this
particle is positive, $\chi=-1$ if negative, and if $\tau_1=\infty$ 
(in which case, no such 
particle is ever drawn) then set $\chi=0$.
For $t\ge0$, let
\begin{align*}
Z'_\lambda(t)&:=e^{-\hat\lambda t}\int e^{\lambda x}\,d\big[\Xi_t^++\hat\Xi_t^+-\Xi_t^--\hat\Xi_t^-\big](x),\\
Z''_\lambda(t)&:=e^{-\hat\lambda t}\int e^{\lambda x}\,d\big[\Xi_t^+-\Xi_t^--\chi\cdot\hat\Xi_t^\circ\big](x).
\end{align*}
By Lemma \ref{lma:processes}, 
it follows that $(Z'_\lambda(t))_{t\ge0}$ and 
$(W_\lambda^+(t)-W_\lambda^-(t))_{t\ge0}$ are equal in distribution, 
and hence by \eqref{eq:conservative_limits} that the limit 
$Z'_\lambda:=\lim_{t\to\infty}Z'_\lambda(t)$ exists almost surely.
In particular,
\begin{equation}\label{eq:Z_distr}
Z'_\lambda
\stackrel{d}{=}
W_\lambda^+-W_\lambda^-.
\end{equation}
Again, by Lemma \ref{lma:processes}, 
either $\tau_1=\infty$ and $Z''_\lambda(t)=0$ for all large $t$, 
or $Z''_\lambda(t)$ evolves from time $\tau_1$ onwards 
again according to the martingale differences of a version of ABBM.
In particular, the limit 
$Z''_\lambda:=\lim_{t\to\infty}Z''_\lambda(t)$ exists almost surely, 
although its not distributed as $W_\lambda^+-W_\lambda^-$.

Next, we introduce notation for the additive martingale associated with 
the marked particles. 
For $t\ge0$, we let
\[
\hat W_\lambda(t)
:=
e^{-\hat\lambda t}\int e^{\lambda x}\,d\big[\hat\Xi_t^++\hat\Xi_t^-+\hat\Xi_t^\circ\big](x).
\]
By Lemma \ref{lma:processes}, on the event that $\tau_1<\infty$, 
this is the additive martingale for a BBM.
Since, by Proposition \ref{prop:mart_lim}, 
the latter almost surely converges to a positive value, 
it follows that
\[
\hat W_\lambda:=\lim_{t\to\infty}\hat W_\lambda(t)
\] 
exists almost surely, 
and that
\begin{equation}\label{eq:W-hat}
\P\big(\hat W_\lambda=0,\tau_1<\infty\big)
=0.
\end{equation}

Recall that one of $\hat\Xi_t^+$ and $\hat\Xi_t^-$ (possibly both) 
will be zero for all $t$.
By the definitions, for all $t\ge0$ we have
\begin{equation}\label{eq:coupling_id}
Z'_\lambda(t) 
= Z''_\lambda(t) + \chi\cdot\hat W_\lambda(t).
\end{equation}
Combining \eqref{eq:W-hat} and~\eqref{eq:coupling_id} 
it follows that
\begin{equation}\label{E_ZZ0}
\P\big(Z'_\lambda=Z''_\lambda=0,\tau_1<\infty\big)
=0.
\end{equation}

Let $B'_\lambda$ be the event that $\Xi'$ survives 
and $Z'_\lambda=0$, and similarly $B''_\lambda$ 
for $\Xi''$ and $Z''_\lambda$.
Since either process surviving implies $\tau_1<\infty$, 
it follows using \eqref{E_ZZ0} that
\begin{equation}\label{eq:Z_bound1}
\P(B'_\lambda) + \P(B''_\lambda) 
\le 1.
\end{equation}
Moreover, by Lemma \ref{lma:distance}, we have that
\begin{equation}\label{eq:Z_bound2}
| \P(B'_\lambda) - \P(B''_\lambda)| 
\le e^{-1}.
\end{equation}
Combining~\eqref{eq:Z_bound1} and~\eqref{eq:Z_bound2} 
with~\eqref{eq:Z_distr}, we conclude that for any non-trivial 
initial configuration $(\Xi_0^+,\Xi_0^-)$, we have
\begin{equation}\label{eq:limit_bound}
\P\big(\{W_\lambda^+=W_\lambda^-\} \cap  \mathcal{S}\big)
= \P\big(B'_\lambda\big) \le \frac{1+e^{-1}}{2}.
\end{equation}

Finally, to conclude the proof, let $(\mathcal{F}_t)_{t\ge0}$ 
denote the filtration in which $\mathcal{F}_t$ 
is the $\sigma$-algebra generated by $\{(\Xi_s^-,\Xi_s^+):0\le s\le t\}$.
By L\'evy's 0--1 law we have, almost surely, that
\[
\lim_{t\to\infty}\P\big(\{W_\lambda^+=W_\lambda^-\}\cap \mathcal{S}\big|\mathcal{F}_t\big)
={\bf 1}_{\{W_\lambda^+=W_\lambda^-\}\cap \mathcal{S}}.
\]
Moreover, by the Markov property of the process, 
it follows from~\eqref{eq:limit_bound} that, almost surely,
\[
\P\big(\{W_\lambda^+=W_\lambda^-\}\cap \mathcal{S}\big|\mathcal{F}_t\big)
=\P\big(\{W_\lambda^+=W_\lambda^-\}\cap \mathcal{S}\big|(\Xi_t^-,\Xi_t^+)\big)\le\frac{1+e^{-1}}{2}.
\]
Therefore $\{W_\lambda^+=W_\lambda^-\}\cap \mathcal{S}$ 
can only occur on an event of measure zero, 
which completes the proof of Proposition \ref{prop:equality} for the additive martingales.

\medskip

For the derivative martingales, the exact same argument works, except that we replace the differences $Z'_\lambda$ and $Z''_\lambda$ for the derivative martingales:
\begin{align*}
Z'(t)&:=e^{-2t}\int (\sqrt2 t - x) e^{\sqrt2 x} \, d\big[\Xi_t^++\hat\Xi_t^+-\Xi_t^--\hat\Xi_t^-\big](x),\\
Z''(t)&:=e^{-2t}\int (\sqrt2 t - x) e^{\sqrt2 x} \, d\big[\Xi_t^+-\Xi_t^--\chi\cdot\hat\Xi_t^\circ\big](x).
\end{align*}
As before, on the event $\tau_1<\infty$ we cannot have that $Z'$ and $Z''$ both tend to 0.
However, the total variation distance between them is bounded away from 1, and on the event that $D^+_\infty = D^-_\infty$, the probability that $Z\to0$ tends to 1 by Levy's 0--1 law.
\qed

\begin{remark}
The bound of $(1+e^{-1})/2$ in~\eqref{eq:limit_bound} is weaker 
(but sufficient) than the bound of $1/2$ obtained
in \cite{AGJM}.
In the context of \cite{AGJM} the two events have the same probability.
It is possible to improve our bound as follows: 
instead of suppressing the first branching event, 
we could randomly suppress 
one of the first $k$ events.
The total variation distance between the 
Poisson process and the slightly thinned Poisson process 
 obtained in this way is $O(k^{-1/2})$.
\end{remark}

\begin{remark}
Although the quantities $W_\lambda^+$ (resp.\ $W_\lambda^-$) 
capture the number of positive and neutral (resp.\ negative and neutral) 
particles with asymptotic speed $\lambda$, 
it follows from Proposition \ref{prop:equality} that, 
when $W_\lambda^+>W_\lambda^-$, 
there are positive particles with asymptotic speed $\lambda$, 
and these particles comprise a positive fraction of all particles 
with asymptotic speed $\lambda$, 
in the sense that their contribution to $W_\lambda^+(t)$, 
for all large $t$, is bounded away from zero.
\end{remark}

\section{Limiting speed of the interface}
\label{sec:interface}

In this section, we analyze the 
ABBM interface. 
We first establish the existence of the limiting speed (Proposition \ref{prop:interface}).
We then examine properties of its law (Propositions \ref{prop:continuity} and \ref{prop:support}).
Together these results imply 
our main result, Theorem \ref{thm:interface} above.

\subsection{Existence of the limiting speed}

\begin{prop}\label{prop:interface}
Consider the ABBM started from a finite, 
non-trivial and ordered initial configuration.
On the event of coexistence, the limit 
\[
\lambda_*
:= \lim_{t\to\infty} \frac{I(t)}{t}
\] 
exists almost surely.
\end{prop}

\begin{proof}
Consider a finite, non-trivial and ordered initial configuration of particles, 
in which the rightmost negative particle is (strictly) to 
the left of the leftmost positive particle.
Let $\mathcal{C}$ denote the event of coexistence, 
i.e., that there are both positive and negative particles present at all times.
According to Theorem \ref{thm:coexistence}, 
$\mathcal{C}$ occurs with positive probability.

Let $\cG_1$ denote the event that for all large $t$, 
all particles present in the system are located in 
the interval $[-\sqrt{2}t,\sqrt{2}t]$.
The argument leading to~\eqref{eq:max_bound} 
shows that $\P(\cG_1)=1$.
On the event $\cC\cap \cG_1$ the interface is 
well-defined and finite at all times, and satisfies, 
for all large enough $t$, 
\begin{equation}\label{eq:interface_bound}
 -\sqrt{2} t \leq I(t) \le \sqrt{2} t. 
 \end{equation}

Let $\cG_2$ denote the event that 
$W_\lambda^+\neq W_\lambda^-$ 
for all rational $\lambda\in(-\sqrt{2},\sqrt{2})$.
Proposition \ref{prop:equality} and countable additivity shows that $\P(\cG_2|\cC)=1$.
Define the sets
\begin{align}
E^+ &:= \big\{\lambda\in\Q : W^+_\lambda > W^-_\lambda \big\},\label{eq:Dpm1}  \\ 
E^- &:= \big\{\lambda\in\Q : W^+_\lambda < W^-_\lambda \big\}.\label{eq:Dpm2}
\end{align}
Then on $\cG_2$, 
we have $E^+\cup E^- = \Q\cap(-\sqrt{2},\sqrt{2})$.

Recall (see \eqref{eq:S_set} above) that, for any $\eps>0$,  
we defined $\alpha_t = t^{1/2+\eps}$ 
and $S_\lambda(t) = [\lambda t-\alpha_t,\lambda t+\alpha_t]$.
Proposition \ref{prop:mart_con} shows that for 
$\lambda\in E^+$ (resp.\ in $E^-$) 
there are eventually positive (resp.\ negative) 
particles in $S_\lambda(t)$.
The fact that the configuration remains ordered 
implies that there cannot be any $\lambda \in E^+$ 
and $\lambda'\in E^-$ with $\lambda < \lambda'$.
Thus on $\cC\cap\cG_1\cap\cG_2$ 
there is some $\mu\in[-\sqrt{2},\sqrt{2}]$ so that
\[ 
E^- 
= \Q\cap(-\sqrt{2},\mu), \qquad E^+ 
= \Q\cap(\mu,\sqrt{2}).
\]

Recall that $I_+(t)$ is the position of the leftmost positive particle.
If $|\mu| < \sqrt{2}$, then taking a rational $\lambda>\mu$ 
arbitrarily close to $\mu$ we find that there are eventually positive 
particles in $S_\lambda(t)$, 
and so $\limsup I_+(t)/t \leq \mu$.
Similarly, $\liminf I_-(t)/t \geq \mu$.
Since $I_-(t)\leq I_+(t)$ this implies the claim.

If $\mu=\sqrt{2}$, then the bound on $I_-(t)$ is the same, 
and we use the bound $I_+(t) \leq \sqrt{2} t$ which holds on $\cG_1$.
The case $\mu=-\sqrt{2}$ is symmetric. 
\end{proof}

\subsection{Properties of the limiting speed}

Next, we show that the law of the limiting speed has no atoms.
We will first show that there are no atoms 
in the interval $(-\sqrt{2},\sqrt{2})$.
This follows by Proposition \ref{prop:equality} and the 
almost sure continuity of the martingale limit $W_\lambda$ in $\lambda$, 
which is due to Biggins~\cite{Big92}.
Afterwards, we shall rule out atoms at $\pm\sqrt{2}$, 
using the derivative martingale and its relation to the edge-behavior of BBM, 
as discussed in Section \ref{S_DMG}. 

\begin{prop}\label{prop:continuity}
Consider ABBM started from a finite, 
non-trivial and ordered initial configuration.
The law of the limiting speed 
\[
\lambda_* 
= \lim_{t\to\infty}\frac{I(t)}{t}
\]
of the interface has no atoms in $(-\sqrt{2},\sqrt{2})$. 
\end{prop}

\begin{proof}
Fix $\mu\in(-\sqrt{2},\sqrt{2})$, 
and let $\cC$ denote the event of coexistence.
Towards a contradiction, 
suppose that $\P(\{\lambda_* = \mu \}\cap\cC)>0$.
Define the sets $E^\pm$ as in \eqref{eq:Dpm1}
and \eqref{eq:Dpm2}. 
Then, on the event $\lambda_*=\mu$, we have that 
$E^- = \Q\cap(-\sqrt{2},\mu)$ and $E^+ = \Q\cap(\mu,\sqrt{2})$.
Since the difference $W_\lambda^+-W_\lambda^-$ 
is almost surely continuous in $\lambda$ (its analytic; see~\cite{Big92}), 
we would have $W^+_\mu = W^-_\mu$. 
However, by Proposition \ref{prop:equality}, this has probability $0$. 
\end{proof}

This approach fails when $\mu=\pm\sqrt{2}$, 
since $W^\pm_{\pm\sqrt{2}} = 0$.
The following proposition completes the 
proof that $\lambda_*$ has no atoms.

\begin{prop}\label{prop:no_sqrt2}
  Consider ABBM started from a finite, 
  non-trivial and ordered initial configuration.
  Then the probability of coexistence 
  with limiting speed $\lambda_* = \pm\sqrt{2}$ is $0$.
\end{prop}

\begin{proof}
  We focus on the event of coexistence with limiting speed 
  $\lambda_* = \sqrt{2}$, showing it has probability 0.
  The case of $-\sqrt{2}$ follows by symmetry.

  By Proposition \ref{prop:equality}, almost surely, 
  for every $\lambda_n = \sqrt{2}-n^{-1}$ 
  we have $W^+_{\lambda_n}\neq W^-_{\lambda_n}$, 
  and additionally $D^+_\infty \neq D^-_\infty$. 
  If, for some $\lambda_n<\sqrt2$, 
  we have $W^+_{\lambda_n} > W^-_{\lambda_n}$ 
  then the limit speed satisfies $\lambda_* \leq \lambda_n$.
  Thus $\lambda_*=\sqrt{2}$ implies 
  $W^+_{\lambda_n} < W^-_{\lambda_n}$ for all the above $\lambda_n$.

  This implies that, on the event $\lambda_*=\sqrt2$, 
  we have that the derivatives $\frac{d}{d\lambda} W^+_\lambda \leq \frac{d}{d\lambda} W^-_\lambda$, 
  when evaluated at $\lambda=\sqrt{2}$.
  By Theorem \ref{thm:madaule} it follows that $D^+_\infty \leq D^-_\infty$, 
  and since the two are unequal, this is strengthened to $D^+_\infty < D^-_\infty$.

Next, we apply Theorem \ref{T_ABK} to the positive and negative particles with $x=0$.
  Let $M^\pm(t)$ denote the location of the rightmost positive/negative particle 
  (possibly the same if they form a neutral particle).
 Put 
 \[
 m(t) = \sqrt2t-\frac{3}{2\sqrt2}\log t.
 \] 
 By Theorem \ref{T_ABK}, 
 the fraction of times in $[0,T]$ for which $M^\pm(t) \leq m(t)$ converges to 
  $\exp(-CD^\pm_\infty)$.
Since $D^+_\infty < D^-_\infty$,  
there is a positive asymptotic density of times at which 
$M^-(t) > m(t)$ and $M^+(t) \leq m(t)$.
 However, even if there is one such time, 
 then the rightmost negative particle has overtaken the rightmost positive particle, 
 thereby ruling out coexistence (since  ABBM is order preserving).
\end{proof}

\medskip

Finally, we show that the limiting speed is fully supported on 
$[-\sqrt2,\sqrt2]$. 
To show this, we construct a specific configuration 
for which $\lambda_*$ is highly likely to be in the 
vicinity of a specific $\lambda\in(-\sqrt2,\sqrt2)$.
This is similar to the proof of Theorem \ref{thm:multi_coex}
in Section \ref{S_multicoex} above, 
where we construct a configuration where each type 
is likely to survive near some $\lambda_i t$. 

\begin{prop}\label{prop:support}
Consider ABBM started from a finite, 
non-trivial and ordered initial configuration. 
The law of the limiting speed 
\[
\lambda_* 
= \lim_{t\to\infty}\frac{I(t)}{t}
\]
of the interface is fully supported on $(-\sqrt{2},\sqrt{2})$. 
\end{prop}

\begin{proof}
Let $\lambda\in(-\sqrt{2},\sqrt{2})$ and $\eps>0$ be given.
By symmetry, let us assume that $\lambda>0$. 
Put
\[
\lambda_\eps^-
=\frac{\lambda-\eps}{1+\eps},
\quad\lambda_\eps^+
=\frac{\lambda+\eps}{1-\eps}.
\]
Without loss of generality, 
we may assume that $\eps>0$ 
is small enough so that 
$0<\lambda_\eps^-<\lambda_\eps^+<\sqrt{2}$.
We will show that
\begin{equation}\label{eq:support}
\P\big(\lambda_\eps^-\le \lambda_*\le \lambda_\eps^+\big)
>0.
\end{equation}

First, we describe specific initial configurations for which 
\eqref{eq:support} holds. 
Fix integers $a,b\ge1$ such that 
\begin{equation}\label{E_eps}
\Big|\frac12\log\left(\frac{a}{b}\right)-\lambda\Big|<\eps.
\end{equation}
Let $\cB$ denote the set of configurations consisting of 
$a n$ negative particles positioned within distance 
$\eps/2$ of $-1$ and $bn$ positive particles positioned 
within distance $\eps/2$ of $+1$.
For starting configurations in $\cB$,
we may interpret the martingale limit 
$W_\lambda^+$ as the sum of 
$bn$ independent martingale limits, 
each of which by Lemma \ref{lem:additivity} 
is distributed as $e^{\lambda x_j}\,W_\lambda$, 
where $x_j\in(1-\eps/2,1+\eps/2)$ 
is the position of the $j$th positive particle.

By the law of large numbers, for all large $n$, 
it follows that 
\begin{align*}
\frac{1}{n}W_{\lambda_\eps^+}^+
&\ge be^{\lambda_\eps^+(1-\eps)}\,\E[W_{\lambda_\eps^+}]
=be^{\lambda+\eps}\,\E[W_{\lambda_\eps^+}],\\
\frac{1}{n}W_{\lambda_\eps^-}^+
&\le be^{\lambda_\eps^-(1+\eps)}\,\E[W_{\lambda_\eps^-}]
=be^{\lambda-\eps}\,\E[W_{\lambda_\eps^-}], 
\end{align*}
with probability at least $3/4$. 
Likewise, for any configuration in $\cB$,
we may interpret $W_\lambda^-$ 
as the sum of $an$ independent martingale limits 
of the form $e^{\lambda x_j}\,W_\lambda$, 
for some $x_j\in(-1-\eps/2,-1+\eps/2)$.
Hence, for all large $n$, 
\begin{align*}
\frac{1}{n}W_{\lambda_\eps^-}^-
&\ge ae^{-(\lambda-\eps)}\,\E[W_{\lambda_\eps^-}],\\
\frac{1}{n}W_{\lambda_\eps^+}^-
&\le ae^{-(\lambda+\eps)}\,\E[W_{\lambda_\eps^+}], 
\end{align*}
with probability at least $3/4$. 
Therefore, for any initial configuration in $\cB$, 
we have, for all large $n$, with probability at least $1/2$, that
\begin{align*}
\frac{1}{n}\big[W_{\lambda_\eps^+}^+-W_{\lambda_\eps^+}^-\big]
&\ge \big(1-\frac{a}{b}e^{-2(\lambda+\eps)}\big)be^{\lambda+\eps}\E[W_{\lambda_\eps^+}]>0,\\
\frac{1}{n}\big[W_{\lambda_\eps^-}^+-W_{\lambda_\eps^-}^-\big]
&\le\big(1-\frac{a}{b}e^{-2(\lambda-\eps)}\big)be^{\lambda-\eps}\E[W_{\lambda_\eps^-}]<0,
\end{align*}
where the final inequalities follow by the choice of $a,b$
in \eqref{E_eps}.
Hence, for every initial configuration in $\cB$, 
with probability at least $1/2$, 
the limiting speed of the interface is contained 
in the interval $[\lambda_\eps^-,\lambda_\eps^+]$. 

To conclude, we note that, from any finite, 
non-trivial and ordered initial configuration, 
with all negative particles to the left of all positive particles, 
there is a positive probability that the configuration at time $t=1$ 
is in the set $\cB$. Therefore, since ABBM is Markovian, 
\eqref{eq:support} follows, as claimed. 
\end{proof}

\section{Generalizations and extensions}\label{sec:open}

In this paper, 
we have studied ABBM on the real line.
There are various generalizations 
where similar questions could be studied.
Several problems which arise naturally 
from this work are discussed below.

\subsection{Other annihilating spatial branching processes}

One could consider a whole range of different 
branching and diffusion mechanisms, 
beyond the dyadic branching and diffusive continuous motion of BBM.
Aspects that can be generalized include:

\begin{description}
\item[Branching] Instead of dyadic splitting we could replace each particle, 
after an exponentially distributed lifetime, by a random number of particles 
positioned randomly on $\R$ according to some reproduction law $\mu$ 
(shifted so that the configuration is centered around the position of the 
particle at the time of the branching event).
The simplest of these is branching into two particles without displacement, 
but the number of children could be random, and some or all of them 
could be created at some displaced location.
\item[Movement] Similarly, the movement of the particles in between 
branching events could more generally be modeled by some L\'evy process, 
which could include both discrete and continuous motion.
This includes the discrete random walk on $\Z$, where jumps are $\pm 1$.
\item[Time] We can also consider processes in discrete time, 
where particles branch at each time step, with some displacement 
for the offspring.
\item[Space] Finally, we can also consider the process on other spaces, 
such as, $\Z^d,\R^d$, or on other graphs or manifolds.
Note that in higher dimension, Brownian motions do not collide almost surely, 
and the rule for annihilation might require adaptation.
In discrete spaces, however, no such change is required.
\end{description}  

We refer to a process of this more general class as an 
{\bf annihilating branching random walk} (ABRW).

In many aspects, we expect that many ABRW, 
under mild assumptions, will behave similarly to ABBM.
However, the methods of this paper are not necessarily 
sufficient even in one dimension.
We focus in what follows in the one-dimensional case.

To apply our proof and address the question of coexistence for ABRW, 
the first step is to extend the theory for additive martingales.
Under suitable tail assumptions on the branching and movement, 
it is still the case that for some $\hat\lambda$ 
(which may no longer be $1+\lambda^2/2$) 
we have additive martingales $W_\lambda(t)$.
It is also expected that these converge to a 
continuous function in $\lambda$ on some interval $I\in\R$.
This is known for some settings, 
such as simple branching random walk on $\Z$, 
but seems not to be proved in full generality.

\begin{question}
For which ABRWs does coexistence occur with positive probability?
\end{question}

Having established coexistence, the question of a limit speed for the interface separating positive and negative particles may be more delicate.
For instance, for an ABRW which is not order preserving, the number of interfaces is no longer decreasing as a function of time, which complicates the analysis.
One example of such a model is any process where the particles jump discontinuously in space.
Another is the model with ``soft'' annihilating, where a pair of particles do not annihilate immediately upon collision, but at some exponential rate in terms of the intersection local time of the pair.

\begin{question}
  If an ABRW is not order preserving, is the number of interfaces almost surely finite? Or can there be infinitely many segments of each type?
\end{question}

We say that an ABRW is \textbf{dominating} 
if after a branching event the number of particles at each 
location is at least the number before the branching event.
For instance, this is the case if branching does not cause displacement, 
or if some children are displaced but at least one stays at the parent's location.
The enhanced coupling construction, used to prove the existence of a 
limit speed for the interface, relies on this domination assumption in a central way.
Specifically, after the first branching event, the system contains an 
extra particle (which is marked) but no particles are missing.
In our proof, domination is used to establish that the limit of the 
additive martingales are almost surely distinct.

\begin{question}
  If an ABRW is not dominating, does the limit speed still exist?
\end{question}

We list here some conditions under which our proofs 
and results hold with no essential modifications.
These include, for example, simple random walk on 
$\Z$ in continuous time with dyadic branching.
\begin{itemize} 
\item The additive martingales 
$W_\lambda(t) = e^{-\hat\lambda t}\sum e^{\lambda X_j(t)}$ 
must converge in (almost surely and in $L^1$) 
to a non-zero limit on some interval $I\subset\R$ 
and to $0$ outside $I$.
\item The location of a typical particle at time $t$ (previously $N(0,t)$) 
needs to satisfy a large deviation principle 
$\P(X\approx at) = e^{-I(a)t + o(t)}$.
\item The dominant contribution to $W_\lambda(t)$ 
comes from particles near $x(\lambda) t$, 
where $x(\lambda)$ is the Legendre transform of $I$, 
namely $x(\lambda)$ maximizes $\lambda x - I(x)$.
This is needed for the proofs that as $x(\lambda)$ maps $I$ 
bijectively to the full range of speeds that exist in the (single-type) 
branching random walk.
\item For the enhanced coupling, 
it is necessary that the branching is dominating.
\item For uniqueness of the interface, 
and existence of its limit speed, 
it is needed that the motion is order-preserving.
\item We used the continuity of the martingale 
limit $W_\lambda$ in order to show that the distribution of the random slope
$\lambda_*$ is continuous. 
\end{itemize}

\subsection{Existence of a limiting speed for multiple types}

Our methods do not quite suffice to establish the 
existence of a limiting speed in settings with more than one interface. 
To see what can go wrong, consider ABBM with 
negative particles initially positioned at $\pm1$ and a 
positive particle at the origin. 
By Theorem \ref{thm:multi_coex} there is positive probability that 
descendants of all three particles initially present survive at all times. 
On this event, there are two interfaces between 
positive and negative particles present at all times. 
Our proof of Theorem \ref{thm:interface} was based on being able to 
relate the limiting speed to a root of the equation 
$W_\lambda^+-W_\lambda^-=0$. 
However, in the setting described above, 
our methods do not alone show that, on the event of coexistence, 
the equation has a solution in $(-\sqrt{2},\sqrt{2})$. 

Next, we formalize the above to the multi-type setting. 
Consider an arbitrary initial configuration consisting of a 
finite number of particles of $m\ge2$ different types, 
positioned so that no two particles of different type occupy the same location. 
At each time $t\ge0$ the configuration of particles gives rise to some 
(random) number $K(t)$ of {\bf interfaces}, 
which (as in the two-type case) are defined as the midpoints of 
maximal vacant open intervals of the real line, 
with endpoints occupied by particles of different types. 
On the other hand, we refer to maximal segments of particles 
of the same type as {\bf blocks}. 
Due to the annihilating feature and the order preserving property, 
the number of interfaces (and blocks) can only decrease as time evolves, 
and so $K(t)$ approaches a limit $K$ as $t\to\infty$. 
Let $C_k$ denote the event that the limiting number of interfaces $K=k$.

The evolution from a $m$-type configuration 
can be coupled with the evolution from a two-type configuration, 
by renaming the blocks of the configuration 
positive/negative in an alternating fashion. 
Note that these two processes evolve identically, 
until the number of interfaces decrease. 
At that point, we must restart the coupling 
(renaming the blocks once again). 
In this way, we repeat this procedure at 
each time the number of interfaces decrease, 
which can happen at most a finite number of times. 
As a consequence, the martingale difference 
$W_\lambda^+(t)-W_\lambda^-(t)$ associated with this 
coupled process has an almost sure limit, 
and Proposition \ref{prop:equality} shows that for fixed 
$\lambda$ the difference is almost surely nonzero.

\begin{question}
Is it true, on the event $C_k$, that each of the $k+1$ 
surviving blocks occupy a linear segment of the real line? 
That is, does each give rise to a 
``positive martingale difference'' in an interval of $\lambda$ values?
\end{question}

If the above holds, then existence of $k$ (distinct) 
limit speeds would follow from Proposition \ref{prop:equality},
and the limiting speeds would be distinct.

\subsection{The higher-dimensional setting}

Similar questions to those examined in this work
can also be asked in higher dimensions.
However, already in two dimensions, 
some care is needed in defining the process, 
as particles with zero radius no longer collide.
One way around this problem is to work with an ABRW on a lattice.

Consider, for instance, an ABRW performing 
simple random walk on $\Z^2$ in continuous time. 
By projecting the particles onto the first coordinate axis, 
we can think of the process as a ABRW on $\Z$, 
given by the first coordinate, 
where each particle also remembers its second coordinate.
Particles only annihilate if they meet and their second coordinates also agree.
The various conditions needed to establish coexistence hold in this case, 
implying positive probability of coexistence in this model.
However, this projection is not order-preserving, 
and so understanding the interface is more delicate.
Indeed, it is not clear even how to define the interface, 
and whether this interface will have an asymptotic limit.

\begin{question}
  In two (and higher) dimensions, does the interface separating 
  the two types have an asymptotic limit shape?
  If so, what can be said about the class of shapes to which that limit belongs?
\end{question}

In analogy with other models for competing growth, 
we expect that a typical description in two (or more) 
dimensions is that of the two types roughly occupying 
complementing regions of some scaled deterministic 
limit shape $U\subset\R^d$.
(The set is the sub-level set of the large deviation 
rate function for the random walks.)
Indeed, applying our arguments using the higher 
dimensional additive martingales 
suggests that for almost every $u\in U$, 
the particles near $tu$ are eventually of a fixed type.
Moreover, it seems possible that one of the two types of 
particles is able to survive despite being eventually 
``surrounded'' by particles of the other type, 
in sharp contrast with other spatial models for competing growth~\cite{HP98}.

\begin{figure}
  \centering
  {\includegraphics[width=.4\textwidth]{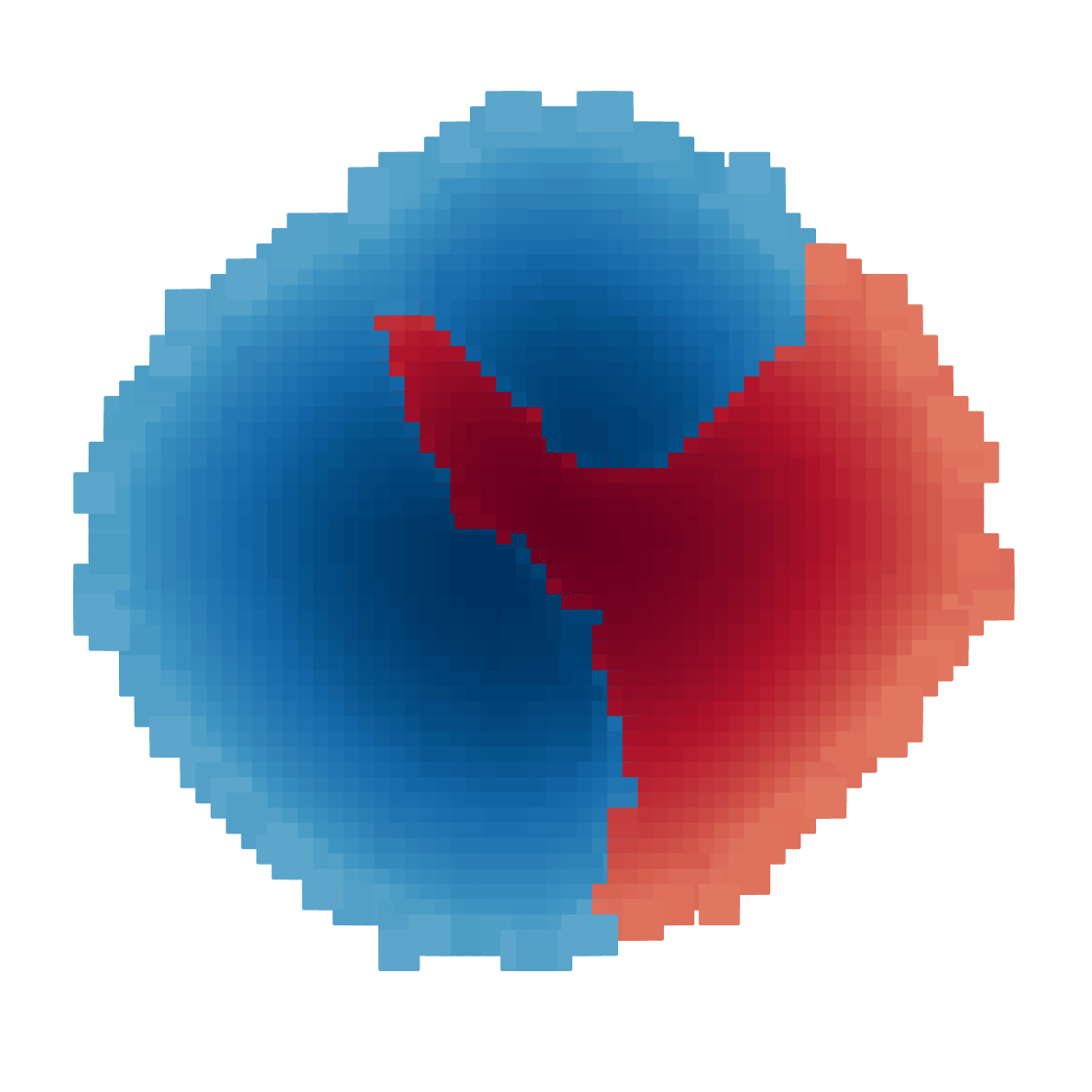}} 
  {\includegraphics[width=.4\textwidth]{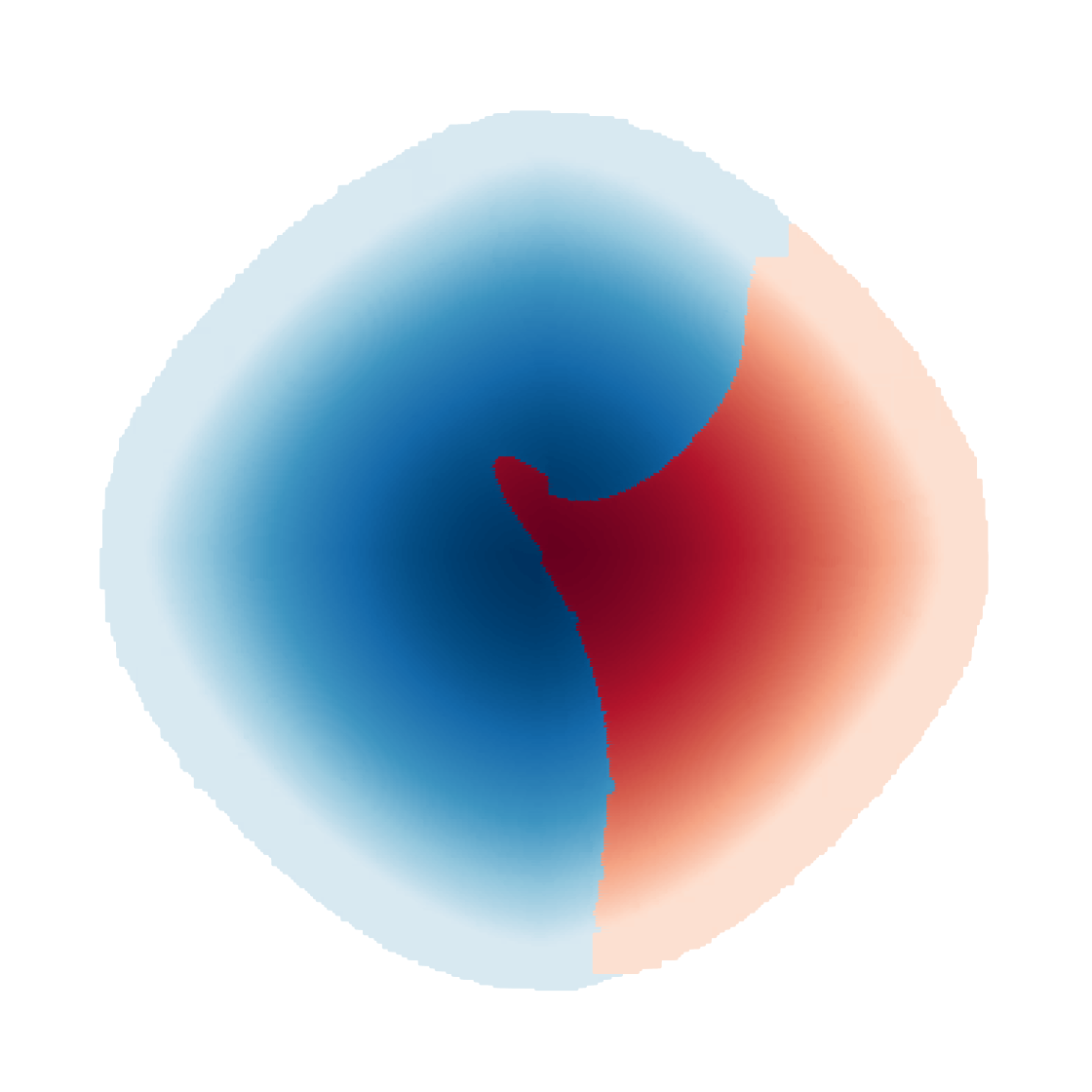}}\\
  {\includegraphics[width=.4\textwidth]{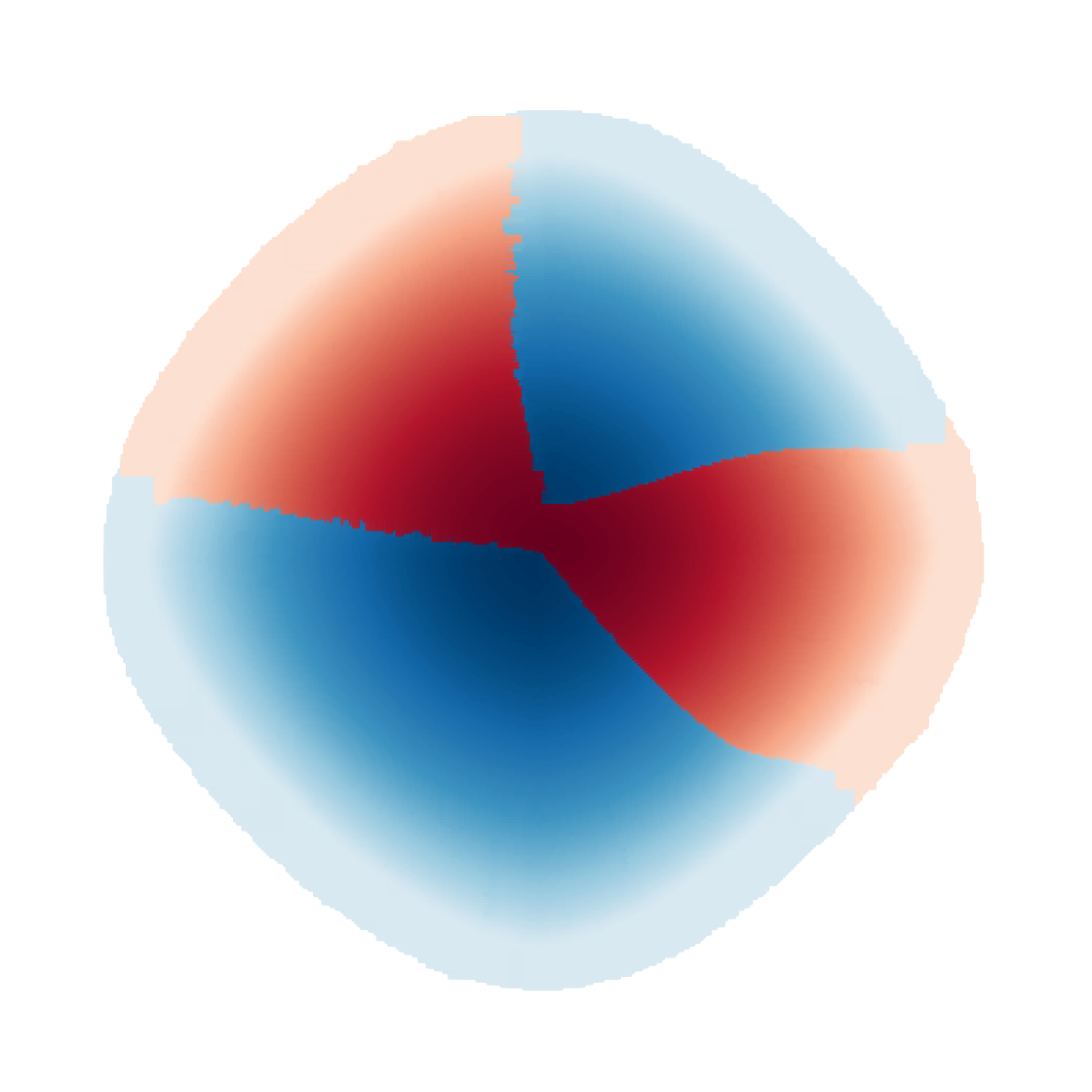}} 
  {\includegraphics[width=.4\textwidth]{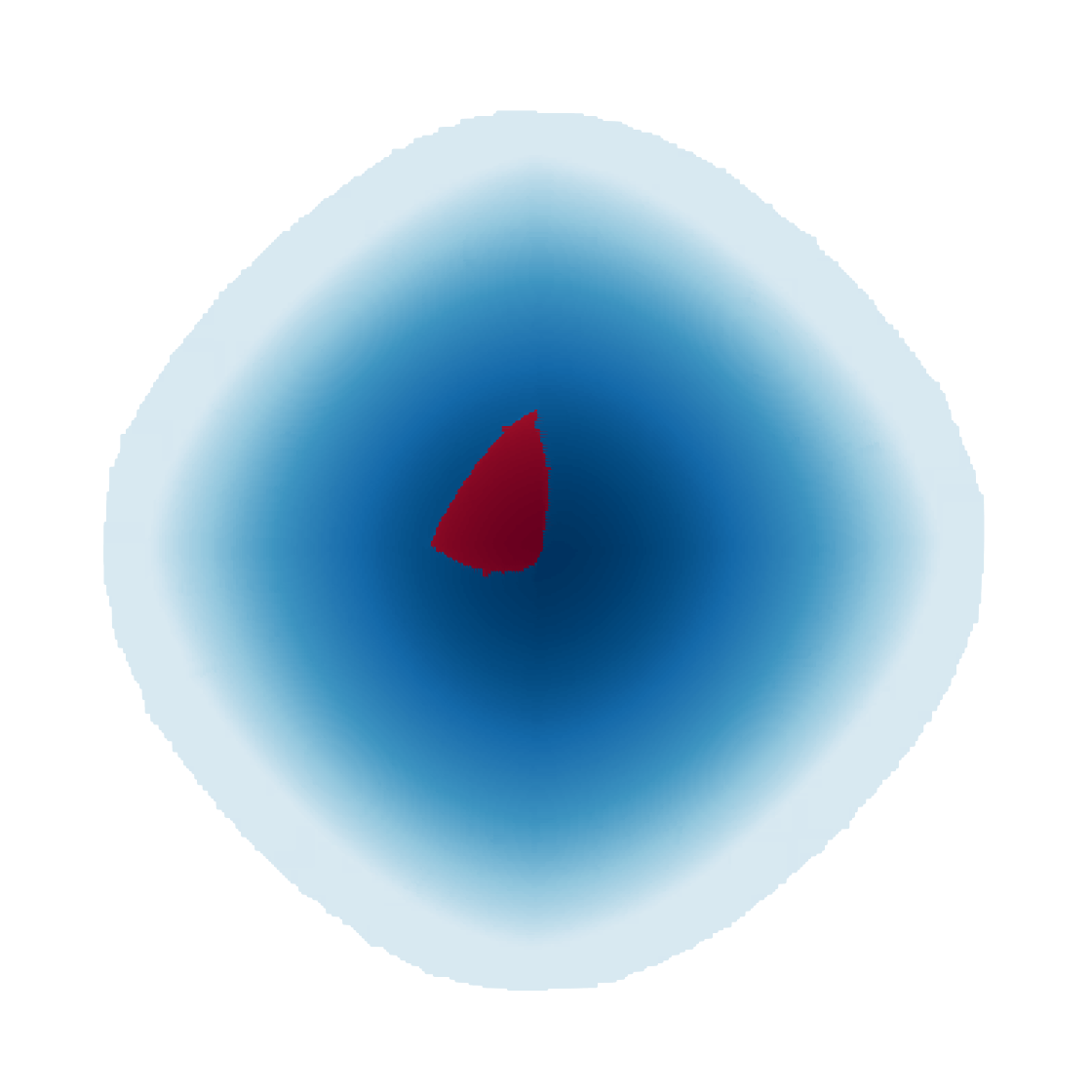}} 
  \caption{Simulations of annihilating branching 
  random walk on $\Z^2$, started with up 2 or 4 particles of each type. 
  The top two are the same process at times 50 and 256.
  The bottom two are distinct runs at time 256.}
  \label{fig:ABRW2d}
\end{figure}

\providecommand{\bysame}{\leavevmode\hbox to3em{\hrulefill}\thinspace}
\providecommand{\MR}{\relax\ifhmode\unskip\space\fi MR }
\providecommand{\MRhref}[2]{%
  \href{http://www.ams.org/mathscinet-getitem?mr=#1}{#2}
}
\providecommand{\href}[2]{#2}

\end{document}